\documentclass[11pt]{amsart}
\usepackage{amsmath}
\usepackage{amssymb}
\usepackage{amscd}

\def\NZQ{\mathbb}               
\def\NN{{\NZQ N}}
\def\QQ{{\NZQ Q}}
\def\ZZ{{\NZQ Z}}
\def\RR{{\NZQ R}}

\def\P{\mathcal P}

\newtheorem{Theorem}{Theorem}[section]
\newtheorem{Lemma}[Theorem]{Lemma}
\newtheorem{Corollary}[Theorem]{Corollary}
\newtheorem{Proposition}[Theorem]{Proposition}
\newtheorem{Remark}[Theorem]{Remark}

\newtheorem{Example}[Theorem]{Example}

\newtheorem{Definition}[Theorem]{Definition}

\let\epsilon\varepsilon
\let\phi=\varphi
\let\kappa=\varkappa

\begin{document}

\title{Valuation semigroups of two dimensional  local rings}
\author{Steven Dale Cutkosky and  Pham An Vinh}
\thanks{The first author was partially supported by NSF}

\address{Steven Dale Cutkosky, Department of Mathematics,
University of Missouri, Columbia, MO 65211, USA}
\email{cutkoskys@missouri.edu}

\address{Pham An Vinh, Department of Mathematics,
University of Missouri, Columbia, MO 65211, USA}
\email{vapnnc@mizzou.edu}

\begin{abstract} We consider the question of when a semigroup is the semigroup
of a valuation dominating a two dimensional noetherian domain, giving some
surprising examples. We give a necessary and sufficient condition for the pair of a semigroup $S$ and a field extension $L/\mathfrak k$ to be the semigroup and
residue field of a valuation dominating a regular local ring $R$ of dimension two with residue field $\mathfrak k$, generalizing the theorem of Spivakovsky for the case when there is no residue field extension.
\end{abstract}

\maketitle
\section{Introduction}
Suppose that $(R,\mathfrak m_R)$ is a Noetherian local ring which is dominated by a valuation $\nu$. The semigroup of
$\nu$ in $R$ is
$$
S^R(\nu)=\{\nu(f)\mid f\in R\setminus\{0\}\}.
$$
$S^R(\nu)$ generates the value group of $\nu$.

In this paper we give a classification of the semigroups and residue field extensions that may be obtained 
by a valuation dominating a  regular local ring of dimension two. Our results are completely general, as we make no further assumptions on the ring or on the residue field extension of the valuation ring. This classification (given in Theorems 
\ref{Theorem3*} and \ref{Corollary4*}) is very simple. 
The classification does not extend to more general rings. 

We give an example showing that the semigroup
of a valuation dominating a normal local ring of dimension two can be quite different from the semigroup of a regular local ring, even on an $A_2$ singularity (Example \ref{Example1}). In \cite{CT1}, \cite{CT2} and \cite{CDK}, we give  examples showing that the semigroups of valuations dominating regular local rings of dimension $\ge 3$ can be very complicated. 
For instance,
in Proposition 6.3 of \cite{CDK}, we show that  there exists a regular local ring $R$ of dimension 3 dominated by a rational rank 1 valuation $\nu$ which has the property that
given $\epsilon>0$, there exists an $i$ such that $\beta_{i+1}-\beta_i<\epsilon$,
where $\beta_0<\beta_1<\cdots$ is the minimal set of of generators of $S^R(\nu)$. In \cite{CT1} and \cite{CT2} we give examples showing that spectacularly strange behavior of the semigroup can occur for a higher rank valuation. The growth of valuation semigroups is however bounded by a polynomial whose coefficients are computed from the multiplicities of the centers of the 
composite valuations on $R$. This is proven in \cite{CT2}.

 The possible value groups $\Gamma$ of a valuation $\nu$ dominating a Noetherian local ring have been extensively studied and classified, including in the papers MacLane
\cite{M}, MacLane and Schilling \cite{MS}, Zariski and Samuel \cite{ZS}, and Kuhlmann \cite{K}. $\Gamma$ can be any ordered abelian group of finite rational rank
(Theorem 1.1 \cite{K}). The semigroup $S^R(\nu)$ is however not well understood, although it is known to encode important
information about the topology and  resolution of singularities of $\mbox{Spec}(R)$ \cite{B}, \cite{BK}, \cite{Z3}, \cite{Z4}, \cite{Ca}, \cite{CS}, \cite{EN}, \cite{Ka},  
\cite{GT}, \cite{Mi}, \cite{T}, \cite{G-P} to mention a few references, and the ideal theory of $R$ \cite{Z1}, \cite{Z2}, \cite{ZS} and its development in  many subsequent papers.

In Sections \ref{RLR1} through \ref{Poly} of this paper we analyze valuations dominating a regular local ring $R$ of dimension two.
 Our analysis is constructive, being based on
an algorithm which finds a generating sequence for the valuation. A generating sequence of $\nu$ in $R$ is a  
set of elements of $R$ whose initial forms are generators of  the graded $\mathfrak k=R/\mathfrak m_R$-algebra $\mbox{gr}_{\nu}(R)$ (Section \ref{Prel}).
The characteristic of the residue field of $R$
does not appear at all in the proofs, although the proof may be simplified significantly if the assumption 
that $R$ has equal characteristic is added; in this case we may reduce to the case where $R$ is a polynomial ring
over a field (Section \ref{Poly}).  A construction of a generating sequence, and the subsequent classification of the  semigroups, is classical in the case when the residue field of $R$ is algebraically closed; this was proven by Spivakovsky in \cite{S}. Besides the complete generality of our results, our proofs differ from those of Spivakovsky in that we
only use elementary techniques, using nothing more sophisticated than the definition of linear independence 
in a vector space, and the definition of the minimal polynomial of an element in a field extension. 
In our proof we construct the residue field of the valuation ring as a tower of primitive extensions; the
minimal polynomials  of the primitive elements are used to construct the generating sequence for the valuation. 
It is not necessary for $R$ to be excellent in our analysis; the only place in this paper where excellence
manifests itself is in the possibility of ramification in the extension of a valuation to the completion of a non excellent regular local ring (Proposition \ref{Prop17}).

In a finite field extension, the quotient of the valuation group of an extension of a valuation
by the value group is always a finite group (2nd corollary  on page 52 of \cite{ZS}).  This raises the following question: Suppose that $R\rightarrow T$ is a finite extension of regular local rings, and $\nu$ is a valuation which dominates $R$. Is $S^T(\nu)$ a finitely
generated module over the semigroup $S^R(\nu)$?  We give a counterexample to this question in Example \ref{Example3}.
This example is especially interesting in light of the results on relative finite generation in the papers \cite{GHK} of Ghezzi, H\`a and Kashcheyeva, and \cite{GK} of Ghezzi and Kashcheyeva.

We now turn to a discussion of our results on regular local rings of dimension two. 
We obtain the following necessary and sufficient condition for a semigroup and field extension to
be the semigroup and residue field extension of a valuation dominating a {\it complete} regular local
ring of dimension two in the following theorem (proven in Section \ref{Proof1}):

\begin{Theorem}\label{Theorem3*} Suppose that $R$ is a complete regular local ring of dimension two with residue field $R/\mathfrak m_R=\mathfrak k$.
Let $S$ be a subsemigroup of the positive elements of a totally ordered abelian group and $L$ be  a field extension of $\mathfrak k$. 
Then $S$ is the semigroup of a valuation $\nu$ dominating $R$ with residue field $V_{\nu}/\mathfrak m_{\nu}=L$ if and only if
there exists a finite or countable index set $I$, of cardinality $\Lambda=|I|-1\ge 1$ and elements $\beta_i\in S$ for $i\in I$ and $\alpha_i\in L$ for $i\in I_+$, where $I_+=\{i\in I\mid i>0\}$,  such that
\begin{enumerate}
\item[1)] The semigroup $S$ is generated by $\{\beta_i\}_{i\in I}$ and 
the field $L$ is generated  over $\mathfrak k$ by $\{\alpha_i\}_{i\in I_+}$.
\item[2)] Let
$$
\overline n_i=[G(\beta_0,\ldots,\beta_i):G(\beta_0,\ldots,\beta_{i-1})]
$$
and
$$
d_i=[\mathfrak k(\alpha_1,\ldots,\alpha_i):\mathfrak k(\alpha_1,\ldots, \alpha_{i-1})].
$$
Then there are inequalities
$$
\beta_{i+1}>\overline n_id_i\beta_i>\beta_i
$$
with $\overline n_i<\infty$ and $d_i<\infty$ for $1\le i<\Lambda$ and if $\Lambda<\infty$, then either $\overline n_{\Lambda}=\infty$ and $d_{\Lambda}=1$ or $\overline n_{\Lambda}<\infty$ and $d_{\Lambda}=\infty$.
\end{enumerate}
\end{Theorem}

Here $G(\beta_0,\ldots,\beta_i)$ is the subgroup  generated by $\beta_0,\ldots,\beta_i$.

The case when $R$ is not complete is more subtle, because of the possibility, when $R$ is not complete,  of the existence of a rank 1
discrete valuation which dominates $R$ and such that the residue field extension $V_{\nu}/\mathfrak m_{\nu}$ of $\mathfrak k=R/\mathfrak m_R$ is finite. For all other valuations $\nu$ which dominate $R$ (so that $\nu$ is not rank 1 discrete with $V_{\nu}/\mathfrak m_{\nu}$ finite over $\mathfrak k$)
the analysis is the same as for the complete case, as there is then a unique extension of $\nu$ to a valuation dominating the completion of $R$ which is an immediate extension; that is, there is no extension of the valuation semigroups or of the residue fields of the valuations. The differences between the complete and non complete cases are
explained in more detail by Theorem \ref{TheoremR4}, Corollary \ref{CorollaryR2}, Example \ref{ExampleR1}, Proposition \ref{Prop17} and Corollary \ref{CorollaryN1} to Theorem \ref{Theorem3*}.

We give a necessary and sufficient condition for a  semigroup  to be the semigroup of a valuation dominating a regular local ring
of dimension two in the following theorem, which is proven in Section \ref{Proof2}:

\begin{Theorem}\label{Corollary4*} 
Suppose that $R$ is a  regular local ring of dimension two.
Let $S$ be a subsemigroup of the positive elements of a totally ordered abelian group. 
Then $S$ is the semigroup of a valuation $\nu$ dominating $R$  if and only if
there exists a finite or countable index set $I$, of cardinality $\Lambda=|I|-1\ge 1$ and elements $\beta_i\in S$  for $i\in I$ such that
\begin{enumerate}
\item[1)] The semigroup $S$ is generated by $\{\beta_i\}_{i\in I}$.
\item[2)] Let 
$$
\overline n_i=[G(\beta_0,\ldots,\beta_i):G(\beta_0,\ldots, \beta_{i-1})].
$$
There are inequalities
$$
\beta_{i+1}>\overline n_i\beta_i
$$
with $\overline n_i<\infty$ for $1\le i<\Lambda$. If $\Lambda<\infty$ then $\overline n_{\Lambda}\le\infty$.
 \end{enumerate}
\end{Theorem}

Theorem \ref{Corollary4*} is proven by Spivakovsky when $R$ has algebraically closed residue field in \cite{S}.

 The proof in Section 5 of \cite{CDK}, given for the case when the residue field of $R$ is algebraically closed,
now extends to arbitrary regular local rings of dimension two, using the conclusions of Theorem \ref{Corollary4*}, to
prove the following:

\begin{Corollary}\label{lim}
Suppose that $R$ is a regular local ring of dimension two and $\nu$ is a rank 1 valuation dominating $R$. Embed the value group of
$\nu$ 
in $\RR_+$ so that $1$ is the smallest nonzero element of $S^R(\nu)$. Let $\phi(n)=|S^R(\nu)\cap (0,n)|$ for $n\in \ZZ_+$.
Then
$$
\lim_{n\rightarrow \infty} \frac{\phi(n)}{n^2}
$$
exists. The set of limits which are obtained by such valuations $\nu$ dominating $R$ is the real half open interval $[0,\frac{1}{2})$.
\end{Corollary}

As a consequence of Theorem \ref{Theorem3*}, we obtain the following example, which we prove in Section \ref{Proof2}, showing the subtlety of the
criteria of Theorem \ref{Theorem3*}.

\begin{Example}\label{Nores} There exists a semigroup $S$ which satisfies the sufficient conditions 1) and 2) of Theorem \ref{Corollary4*}, such that if $(R,\mathfrak m_R)$ is a 2-dimensional regular local ring dominated by a valuation $\nu$ such that $S^R(\nu)=S$, then $R/\mathfrak m_R=V_{\nu}/\mathfrak m_{\nu}$; that is, there can be no
residue field extension.
\end{Example}

The main technique we use in the proofs of the above theorems is the algorithm of Theorem \ref{Theorem1*}, which 
constructs a sequence of elements $\{P_i\}$ in $R$, starting with a given regular system of parameters
$P_0=x$, $P_1=y$ of  $R$, which gives a generating sequence of $\nu$ in $R$. This fact is proven in Theorems  \ref{Corollary1*} and \ref{Corollary3*}.

In Section \ref{RLR2}, we develop the birational theory of the generating sequence $\{P_i\}$, generalizing to the case when
$R$ has arbitrary residue field the results of \cite{S}.

Suppose that $R$ is a regular local ring of dimension two which is dominated by a valuation $\nu$. Let $\mathfrak k=R/\mathfrak m_R$ and 
\begin{equation}\label{eqX3*}
R\rightarrow T_1\rightarrow T_2\rightarrow \cdots 
\end{equation}
be the sequence of quadratic transforms along $\nu$, so that $V_{\nu}=\cup T_i$, and
$L=V_{\nu}/\mathfrak m_{\nu}=\cup T_i/\mathfrak m_{T_i}$.
Suppose that $x,y$ are regular parameters in $R$, and 
let $P_0=x$, $P_1=y$ and $\{P_i\}$ be the sequence of elements of $R$ constructed in Theorem \ref{Theorem1*}.
Suppose there exists some smallest value $i$ in the sequence (\ref{eqX3*}) such that
the divisor of $xy$ in $\mbox{Spec}(T_i)$ has only one component. Let $R_1=T_i$. By Theorem \ref{birat}, a local equation of the exceptional divisor and  a strict transform of $P_2$ in $R_1$ are a regular system of  parameters in $R_2$, and a local equation of the exceptional divisor and    a strict transform of $P_i$ in $R_1$ for $i\ge 2$ satisfy the conclusions
of Theorem \ref{Theorem1*} on $R_2$. 

We can repeat this construction, for this new sequence, to construct a sequence of quadratic transforms $R_1\rightarrow R_2$ such that a local equation of the exceptional divisor and a strict transform of $P_3$ is a regular system of parameters
in $R_2$, and a local equation of the exceptional divisor and a strict transform of $P_i$ for $\ge 3$ satisfy the conclusions
of Theorem \ref{Theorem1*} on $R_2$.

We  thus have a  sequence of iterated quadratic transforms  
$$
R\rightarrow R_1\rightarrow R_2\rightarrow \cdots
$$
such that $V_{\nu}=\cup R_i$ and where  a local equation of the exceptional divisor of $R_i\rightarrow R_{i+1}$ and the strict transform of $P_{i+1}$
are a regular system of parameters in $R_i$ for all $i$.

The notion of a generating sequence of a valuation already can be recognized in the famous algorithm of Newton
to find the branches of a (characteristic zero) plane curve singularity. In more modern times, it has been developed
by Maclane \cite{M} (``key polynomials''), Zariski \cite{Z1},  Abhyankar \cite{Ab3}, \cite{Ab4} (``approximate roots''), and Spivakovsky \cite{S}. Most recently, the construction and application of generating sequences  of
a valuation have appeared in many papers, including \cite{CG}, \cite{CGP}, \cite{CP}, \cite{EH}, \cite{FJ}, \cite{GAS}, \cite{GHK}, \cite{GK},
\cite{LMSS}, \cite{Mo}, \cite{V}.  The theory of generating sequences in regular local rings of dimension two is closely related to the configuration of exceptional curves appearing in the sequence of quadratic transforms along the center of the valuation. This subject has been explored in many papers, including \cite{Ca} and \cite{Li}.
The extension of valuations to the completion of a  local ring, which becomes extremely difficult
in higher dimension and rank, is studied in \cite{S}, \cite{HS}, \cite{CE}, \cite{C}, \cite{CK}, \cite{CoG}, \cite{CE} and \cite{HOST}.
There is an extensive literature on  the theory of complete ideals in local rings, beginning with 
Zariski's articles \cite{Z1} and \cite{ZS}.

We thank Soumya Sanyal for his meticulous reading of  this paper.

\section{Preliminaries}\label{Prel} Suppose that $(R,\mathfrak m_R)$ is a Noetherian local domain and $\nu$ is valuation of the quotient field which dominates $R$. Let $V_{\nu}$ be the valuation ring of $\nu$, and $\mathfrak m_{\nu}$ be its maximal ideal. 
Let $\Gamma_{\nu}$ be the value group of $\nu$. Let $\mathfrak k=R/\mathfrak m_R$. The semigroup of $\nu$ on $R$ is
$$
S^R(\nu)=\{\nu(f)\mid f\in R\setminus\{0\}\}.
$$

For $\phi\in \Gamma_{\nu}$, define valuation ideals
$$
\P_{\phi}(R)=\{f\in R\mid \nu(f)\ge \phi\},
$$
and
$$
\P_{\phi}^+(R)=\{f\in R\mid \nu(f)> \phi\}.
$$
We have that $\P_{\phi}^+(R)=\P_{\phi}(R)$ if and only if $\phi\not\in S^R(\nu)$.
The associated graded ring of $\nu$ on $R$ is
$$
{\rm gr}_{\nu}(R)=\bigoplus_{\phi\in\Gamma_{\nu}}\P_{\phi}(R)/\P_{\phi}^+(R).
$$
Suppose that $f\in R$ and $\nu(f)=\phi$. Then the initial form of $f$ in ${\rm gr}_{\nu}(R)$
is 
$$
{\rm in}_{\nu}(f)= f+\P_{\phi}^+(R)\in [{\rm gr}_{\nu}(R)]_{\phi}=\P_{\phi}(R)/\P_{\phi}^+(R).
$$

A set of elements $\{F_i\}_{i\in I}$ such that $\{{\rm in}_{\nu}(F_i)\}$ generates ${\rm gr}_{\nu}(R)$ as a $\mathfrak k$-algebra
is called a generating sequence of $\nu$ in $R$.

We have that the vector space dimension
$$
\mbox{dim}_{R/\mathfrak m_R}\P_{\phi}(R)/\P_{\phi}^+(R)<\infty
$$
and
$$
\mbox{dim}_{R/\mathfrak m_R}\P_{\phi}(R)/\P_{\phi}^+(R)\le [V_{\nu}/\mathfrak m_{\nu}:R/\mathfrak m_R]
$$
for all $\phi\in \Gamma_{\nu}$.

$S^R(\nu)$ is countable and is well ordered of ordinal type $\le \omega^2$ by Proposition 2, Appendix 3 \cite{ZS}. 
Further, $V_{\nu}/\mathfrak m_{\nu}$ is a countably generated field extension of $\mathfrak k=R/\mathfrak m_R$, since
${\rm gr}_{\nu}(R)$ is a countably generated vector space over $R/\mathfrak m_R$, and if $0\ne \alpha\in V_{\nu}/\mathfrak m_{\nu}$, then $\alpha$ is the residue of $\frac{f}{g}$ for some  $f,g\in R$ with $\nu(f)=\nu(g)$.

We will make use of Abhyankar's Inequality (\cite{Ab1}, Appendix 2 \cite{ZS}):
\begin{equation}\label{eq17}
{\rm\, rat\,\, rank\, }\nu+{\rm\, trdeg}_{R/\mathfrak m_R} V_{\nu}/\mathfrak m_{\nu}\le {\rm\, dim\, }R
\end{equation}
If equality holds then $\Gamma_{\nu}\cong \ZZ^m$ as an {\it unordered} group, where $m={\rm rat\,\,rank\,}\nu$, and $V_{\nu}/\mathfrak m_{\nu}$ is a finitely generated field extension of  $R/\mathfrak m_R$.

We have that 
$$
{\rm rank\,}\nu\le {\rm rat\,\,rank\,}\nu\le{\rm dim\,}R.
$$
Let $n={\rm rank\,}\nu$. Then we have an order preserving embedding 
\begin{equation}\label{eqZ11}
\Gamma_{\nu}\subset  \Gamma_{\nu}\RR\cong(\RR^n)_{\rm lex}
\end{equation}
(Proposition 2.10 \cite{Ab2}).
We  say that $\nu$ is discrete if $\Gamma_{\nu}$ is discrete in the Euclidean topology.

If $I$ is an ideal in $R$, we may define $\nu(I)=\mbox{min}\{\nu(f)\mid f\in I\setminus \{0\}\}$,
 since $S^R(\nu)$ is well ordered.

$\NN$ denotes the natural numbers $\{0,1,2,\ldots\}$ and $\ZZ_+$ denotes the positive integers $\{1,2,3,\ldots\}$.

Given elements $z_1,\ldots, z_n$ in a group $G$, let
$G(z_1,\ldots,z_n)$ be the subgroup generated by $z_1,\ldots,z_n$. Let $S(z_1,\ldots,z_n)$ be the semigroup
generated by $z_1,\ldots, z_n$.

\begin{Lemma}\label{Lemma2}
Suppose that $\Gamma$ is a totally ordered abelian group, $I$ is a finite or countable index set of cardinality $\ge 2$ and $\beta_i\in\Gamma$ are positive elements for $i\in I$. Let $\Lambda=|I|-1$. Let
$$
\overline n_i=[G(\beta_0,\ldots, \beta_i):G(\beta_0,\ldots,\beta_{i-1})]\in\ZZ_+\cup\{\infty\}
$$
for $\ge 1$.
Assume that $\overline n_i\in\ZZ_+$ if $i<\Lambda$. Let $s_i$ be the smallest positive integer $t$ such that $t\beta_i\in S_{i-1}$ (or $s_i=\infty$ if $i=\Lambda$ and no such $t$ exists).
\vskip .2truein

Suppose that $1\le k<\Lambda$ and $\overline n_i\beta_i<\beta_{i+1}$ for $1\le i\le k-1$. Then
\begin{enumerate}
\item[1)] $s_i=\overline n_i$ for $1\le i\le k$.
\item[2)] If $\gamma\in G(\beta_0,\ldots,\beta_k)$ and $\gamma\ge \overline n_k\beta_k$ then $\gamma\in S(\beta_0,\ldots,\beta_k)$.
\end{enumerate}
\end{Lemma}

\begin{proof} We first prove 2). By repeated Euclidean division, we obtain an expansion $\gamma=a_0\beta_0+a_1\beta_1+\cdots+a_k\beta_k$ with $a_0\in \ZZ$ and $0\le a_i< \overline n_i$ for $1\le i\le k$. Now we calculate, using the inequalities $\overline n_i\beta_i<\beta_{i+1}$,
$$
a_1\beta_1+\cdots+a_k\beta_k<\overline n_k\beta_k.
$$
Thus $a_0>0$ and $\gamma\in S(\beta_0,\ldots,\beta_k)$.

 Now 1) follows from 2) and induction on $k$. 
\end{proof}

A Laurent monomial in $H_0,H_1,\ldots, H_l$ is a product $H_0^{a_0}H_1^{a_1}\cdots H_l^{a_l}$ with $a_0,a_1,\ldots, a_l\in\ZZ$.

Suppose that $R$ is a regular local ring with maximal ideal $\mathfrak m_R$. Suppose that $f\in R$. Then we define
$$
{\rm ord}(f)=\max\{n\in\NN\mid f\in \mathfrak m_R^n\}.
$$

\section{Regular local rings of dimension two}\label{RLR1}

Suppose that $(R,\mathfrak m_R)$ is a Noetherian local domain of dimension two.
Up to order isomorphism, the value groups $\Gamma_{\nu}$ of a valuation $\nu$ which dominates $R$ are by Abhyankar's inequality and Example 3, Section 15, Chapter VI \cite{ZS}:
\begin{enumerate}
\item[1.] $\alpha\ZZ+\beta\ZZ$ with $\alpha,\beta\in\RR$ rationally independent.
\item[2.] $(\ZZ^2)_{\rm lex}$.
\item[3.] Any subgroup of $\QQ$.
\end{enumerate}

Suppose that $N$ is a field, and $V$ is a valuation ring of $N$.
We say that the rank of $V$ increases under completion if there exists an analytically normal local domain
$T$ with quotient field  $N$ such that $V$ dominates $T$ and there exists an extension of $V$ to a valuation ring of the quotient field of $\hat T$ which dominates $\hat T$ and which has higher rank than the rank of $V$.

\begin{Theorem}\label{TheoremR4}(Theorem 4.2, \cite{CK}; \cite{S} in the case when $R/\mathfrak m_R$ is algebraically closed)
Suppose that $V$ dominates an excellent two dimensional local ring $R$. Then the rank of $V$ increases under completion if and only if $V/\mathfrak m_V$ is finite over $R/\mathfrak m_R$  and $V$ is discrete of rank 1.\end{Theorem}

\begin{Corollary}\label{CorollaryR2} If $R$ is complete and $\nu$ is a discrete
rank one valuation which dominates $R$ then $[V_{\nu}/\mathfrak m_{\nu}:R/\mathfrak m_R]=\infty$. 
\end{Corollary}

The following example shows an important distinction between the case when $R$ is complete and when $R$ is not.

\begin{Example}\label{ExampleR1} Suppose that $\mathfrak k$ is a field and $R=\mathfrak k[x,y]_{(x,y)}$ is a localization of a polynomial ring
in two variables. Then there exists a rank one discrete valuation $\nu$ dominating $R$ such that $V_{\nu}/\mathfrak m_{\nu}=\mathfrak k$.
\end{Example}

\begin{proof} Let $f(t)\in\mathfrak k[[t]]$ be a transcendental element over $\mathfrak k(t)$. Embed $R$ into $\mathfrak k[[t]]$ by
substituting $t$ for $x$ and $f(t)$ for $y$. The valuation $\nu$ on $R$ obtained by restriction of the $t$-adic valuation
to $R$ has the desired properties.
\end{proof}

Suppose that $\nu$ is a valuation which dominates $R$. Let $a$ be the smallest positive element in $S^R(\nu)$.
Suppose that $\{f_i\}$ is a Cauchy sequence in $R$ (for the $\mathfrak m_R$-adic topology). Then either there exist
$n_0\in \ZZ_+$, $m\in\ZZ_+$ and $\gamma\in S^R(\nu)$ such that $\gamma<ma$ and $\nu(f_i)=\gamma$ for $i\ge n_0$, or
\begin{equation}\label{eqZ12}
\mbox{Given $m\in \ZZ_+$, there exists $n_0\in\ZZ_+$ such that $\nu(f_i)>ma$ for $i>n_0$}
\end{equation}
Let $I_{\hat R}$ be the set of limits of Cauchy sequences $\{f_i\}$ satisfying (\ref{eqZ12}). Then $I_{\hat R}$ is a prime
ideal in $\hat R$ (\cite{C}, \cite{CG}, \cite{CE}, \cite{S}, \cite{T}). The following proposition is well known.

\begin{Proposition}\label{Prop17}
Suppose that $R$ is a regular local ring of dimension two,
 and let $\nu$ be a valuation which dominates  $R$. 
Then there exists  an extension of $\nu$ to a valuation $\hat \nu$ which dominates the
completion $\hat R$ of $R$ with respect to $\mathfrak m_R$, which has one of the following semigroups:
\begin{enumerate}
\item[1.]  ${\rm rank\,}\nu={\rm rank\,}\hat\nu=1$ and
\begin{equation}\label{eq14}
S^R(\nu)=S^{\hat R}(\hat\nu).
\end{equation}
\item[2.]  $\nu$ is discrete of rank 1,  $\hat\nu$ is discrete of rank 2 and
\begin{equation}\label{eq15}
\mbox{$S^{\hat R}(\hat\nu)$ is generated by $S^R(\nu)$ and an element $\alpha$ such that $\alpha>\gamma$ for all $\gamma\in S^R(\nu)$}.
\end{equation}
\item[3.]  $\nu$ and $\hat\nu$ are discrete of rank 2,  there exists a height one prime $I_R$ in $R$,
 and a discrete rank 1 valuation $\overline \nu$ which dominates the maximal ideal $\mathfrak m_R( R/I_R)$ of $R/I_R$ such that  
\begin{equation}\label{eq16}
\begin{array}{l}
\mbox{$S^{R}(\nu)$ is generated by $S^{R/I_R}(\overline \nu)$ and an element $\alpha$ such that $\alpha>\gamma$}\\
\mbox{for all $\gamma\in S^{R/I_R}(\overline \nu)$.}\\
\mbox{$S^{\hat R}(\hat\nu)$ is generated by $S^{R/I_R}(\overline \nu)$ and an element $\beta$ such that $\alpha-t\beta\in S^{R/I_R}(\overline\nu)$,}\\
\mbox{for some $t\in\ZZ_+$. If $R_{\mathfrak m}$ is excellent, then $t=1$.}
\end{array}
\end{equation}
\item[4.] $\nu$ and $\hat\nu$ are discrete of rank 2, $I_{\hat R}=(0)$ and $S^R(\nu)=S^{\hat R}(\hat \nu)$.
\end{enumerate}
\end{Proposition}

\begin{proof}
First suppose that $\nu$ has rank 1. Then  $I_{\hat R}\cap R=(0)$, so we have an embedding
$R\subset \hat R/I_{\hat R}$. We can then extend $\nu$ to a valuation $\overline\nu$ which  dominates  $\hat R/I_{\hat R}$ by defining for
$f\not\in I_{\hat R}$, $\overline\nu(f+I_{\hat R})=\lim_{i\rightarrow\infty}\nu(f_i)$, where $\{f_i\}$ is a Cauchy sequence in $R$ representing $f$. We have that $S^R(\nu)=S^{\hat R/I_{\hat R}}(\overline \nu)$.

If $I_{\hat R}=(0)$ then we have constructed the desired extension $\hat\nu=\overline\nu $ of $\nu$ to $\hat R$. Suppose that $I_{\hat R}\ne (0)$. 
Then $\hat R/I_{\hat R}$ has dimension 1, so $\overline \nu$ is discrete of rank 1. We have that $I_{\hat R}=(v)$ is a height one prime ideal. We can extend $\overline \nu$ to a rank 2 valuation $\hat\nu$ which dominates $\hat R$ by defining $\hat \nu(f)=(n,\overline\nu(g))\in(\ZZ\bigoplus\Gamma_{\overline \nu})_{\rm lex}$
if $f\in \hat R$ has a factorization $f=v^ng$ where $n\in \NN$ and $v\not\,\mid g$.

Now assume that $\nu$ has rank 2. Further assume that $I_{\hat R}\cap R\ne (0)$. Then $\nu$ has rank 2, and $I_R=I_{\hat R}\cap R$ is a height one prime ideal in $R$. Thus there exists an irreducible $g\in R$ such that $I_R=(g)$. We then have that $I_{\hat R}$ is a height one prime ideal in $\hat R$, so there exists an irreducible $v\in \hat R$ such that $I_{\hat R}=(v)$.

There exists a valuation $\overline\nu$ dominating $R/I_R$ such that if $f\in R$ has a factorization $f=g^nh$ where $g\not\,\mid h$, then 
$$
\nu(f)=n\nu(g)+\overline\nu(h).
$$

Write $g=v^t\phi$ where $t\in\ZZ_+$ and $v\not\,\mid\phi$.
Thus $\phi\not\in I_{\hat R}$. If $R$ is excellent,
then  $g$ is reduced in $\hat R$ (by Scholie IV 7.8.3 (vii) \cite{G}), so $t=1$. 
 We have an inclusion $R/I_R\subset \hat R/I_{\hat R}$, and $\overline\nu$ extends to a valuation $\hat{\overline\nu}$ which dominates $\hat R/I_{\hat R}$. We then extend $\nu$ to a valuation $\hat\nu$ which dominates $\hat R$ by setting
 $$
 t\hat\nu(v)=\nu(g)-\hat{\overline\nu}(\phi)
 $$
 in $\Gamma_{\nu}\RR\cong(\RR^2)_{\rm lex}$.
 Suppose that $0\ne f\in \hat R$. Factor $f$ as $f=v^nh$ where $n\in\NN$ and $v\not\,\mid h$. 
 Then define 
 $$
 \hat\nu(f)=n\hat\nu(v)+\hat{\overline\nu}(h).
 $$
 
 We now show that $S^{R/I_R}(\overline\nu)=S^{\hat R/I_{\hat R}}(\hat{\overline\nu})$. 
 We have that 
 $\hat{\overline\nu}(\mathfrak m(\hat R/I_{\hat R}))=\overline\nu(\mathfrak m(R/I_{R}))$. 
 Suppose that $0\ne h\in \hat R/I_{\hat R}$, and that $\hat{\overline \nu}(h)=\gamma$.
 There exists $n\in \ZZ_+$ such that 
 $n\hat{\overline \nu}(\mathfrak m(\hat R/I_{\hat R}))>\gamma$ and there exists $f\in R$ such that if $\overline f$ is the image of $f$ in $R/I_{R}$, then $\overline f-h\in \mathfrak m^n(\hat R/I_{\hat R})$. Thus $\nu(f)=\overline\nu(\overline f)=\hat{\overline \nu}(h)=\gamma$. 
 
 Suppose that $\mbox{rank }\nu=2$ and $I_{\hat R}\cap R=(0)$. We can extend $\nu$ to a valuation $\overline\nu$ dominating 
 $R/I_{\hat R}$ by defining for $f\not\in I_{\hat R}$, $\overline \nu(f+I_{\hat R})=\lim_{i\rightarrow \infty}\nu(f_i)$ if $\{f_i\}$ is a Cauchy sequence in $R$ converging to $f$. We must have that $I_{\hat R}=(0)$, since otherwise we would be able to extend $\overline \nu$ to a valuation $\tilde\nu$ dominating $\hat R$ which is composite with the rank 2 extension $\overline\nu$ of $\nu$ to
 $\hat R/I_{\hat R}$; this extension would have rank $\ge 3$ which is impossible by Abhyankar's inequality. Thus $I_{\hat R}=(0)$.

 \end{proof} 
 
 \begin{Remark} Nagata gives an example in the Appendix to \cite{Na} of a regular local ring $R$ of dimension two with an irreducible element $f\in R$ such that $f$ is not reduced in $\hat R$.
 \end{Remark}

\section{The Algorithm}

In this section, we will suppose that $R$ is a regular local ring of dimension two, with maximal ideal $\mathfrak m_R$ and residue field $\mathfrak k=R/\mathfrak m_R$.
For $f\in R$, let $\overline f$ or $[f]$ denote the residue of $f$ in $\mathfrak k$.
Suppose that $CS$ is a coefficient set of $R$.
A coefficient set of $R$ is a subset $CS$ of $R$ such that the mapping $CS\rightarrow \mathfrak k$ defined by $s\mapsto \overline s$ is a bijection.
We further require that $0\in CS$ and $1\in CS$.

\begin{Remark}\label{RemarkG4} Suppose that  $x,y$ are regular parameters in $R$, $a,b\in CS$  and $n\in 
\ZZ_+$. Let $c\in CS$ be defined by $\overline{a+b}=\overline c$. Then there exist
$e_{ij}\in CS$  such that 
$$
a+b=c+\sum_{i+j=1}^{n-1}e_{ij}x^iy^j+h
$$
with $h\in \mathfrak m_R^n$. Let $d\in CS$ be defined by $\overline{ab}=\overline d$. Then there exist $g_{ij}\in CS$ such that
$$
ab=d+\sum_{i+j=1}^{n-1}g_{ij} x_iy_j+h'
$$
with $h'\in \mathfrak m_R^n$.
\end{Remark}

\begin{Theorem}\label{Theorem1*} Suppose that  $\nu$ is a  valuation 
of the quotient field of $R$ dominating $R$. Let $L=V_{\nu}/m_{\nu}$ be the residue field of the valuation ring $V_{\nu}$ of $\nu$. For $f \in V_{\nu}$, let $[f]$ denote the class of $f$ in $L$. Suppose that $x,y$ are regular parameters in $R$.
Then  there exist $\Omega\in\ZZ_+\cup\{\infty\}$ and 
$P_i\in \mathfrak m_R$ for $i\in\ZZ_+$ with $i<\min\{\Omega+1,\infty\}$
 such that $P_0=x$, $P_1=y$ and for $1\le i<\Omega$, there is an expression
 \begin{equation}\label{eq11*} 
 P_{i+1} = P_i^{n_i}+\sum_{k=1}^{\lambda_i} c_kP_0^{\sigma_{i,0}(k)}P_1^{\sigma_{i,1}(k)}\cdots P_{i}^{\sigma_{i,i}(k)}
 \end{equation}
 with $n_i\ge 1$, $\lambda_i\ge 1$, 
 \begin{equation}\label{eq12*}
 0\ne c_k\in CS
 \end{equation}
  for $1\le k\le \lambda_i$,
 $\sigma_{i,s}(k)\in\NN$ for all $s,k$,  $0\le \sigma_{i,s}(k)<n_s$ for $s\ge 1$.
 Further,
 $$
 n_i\nu(P_i)=\nu(P_0^{\sigma_{i,0}(k)}P_1^{\sigma_{i,1}(k)}\cdots P_{i}^{\sigma_{i,i}(k)})
 $$
 for all $k$.
 
 For all $i\in\ZZ_+$ with $i<\Omega$, the following are true:
 \begin{enumerate}
 \item[1)] $\nu(P_{i+1})>n_i\nu(P_i)$.
 \item[2)] Suppose that $r\in\NN$, $m\in \ZZ_+$, $j_k(l)\in\NN$  for $1\le l\le m$ and $0\le j_k(l)<n_k$ for $1\le k\le r$ are such that  $(j_0(l),j_1(l),\ldots,j_r(l))$ are distinct for $1\le l\le m$, and 
 $$
 \nu(P_0^{j_0(l)}P_1^{j_1(l)}\cdots P_r^{j_r(l)})=\nu(P_0^{j_0(1)}\cdots P_r^{j_r(1)})
 $$
 for $1\le l\le m$.
 Then
 $$
 1,\left[\frac{P_0^{j_0(2)}P_1^{j_1(2)}\cdots P_r^{j_r(2)}}{P_0^{j_0(1)}P_1^{j_1(1)}\cdots P_r^{j_r(1)}}\right],
 \ldots,
 \left[\frac{P_0^{j_0(m)}P_1^{j_1(m)}\cdots P_r^{j_r(m)}}{P_0^{j_0(1)}P_1^{j_1(1)}\cdots P_r^{j_r(1)}}\right]
 $$
 are linearly independent over $\mathfrak k$.
 \item[3)] Let 
 $$
 \overline n_i=[G(\nu(P_0),\ldots, \nu(P_i)):G(\nu(P_0),\ldots, \nu(P_{i-1}))].
 $$
  Then $\overline n_i$ divides $\sigma_{i,i}(k)$ for all $k$ in (\ref{eq11*}). In particular, $n_i=\overline n_id_i$ with $d_i\in \ZZ_+$ 
 \item[4)] There exists $U_i=P_0^{w_0(i)}P_1^{w_1(i)}\cdots P_{i-1}^{w_{i-1}(i)}$ for $i\ge 1$ with $w_0(i),\ldots, w_{i-1}(i)\in\NN$ 
 and $0\le w_j(i)<n_j$ for $1\le j\le i-1$ such that
 $\nu(P_i^{\overline n_i})=\nu(U_i)$ and if 
 $$
 \alpha_i=\left[\frac{P_i^{\overline n_i}}{U_i}\right]
 $$
 then 
 $$
 b_{i,t}=\left[\sum_{\sigma_{i,i}(k)=t\overline n_i}c_k\frac{P_0^{\sigma_{i,0}(k)}P_1^{\sigma_{i,1}(k)}\cdots P_{i-1}^{\sigma_{i,i-1}(k)}} {U_i^{(d_i-t)}}\right]\in \mathfrak k(\alpha_1,\ldots,\alpha_{i-1})
 $$
 for $0\le t\le d_i-1$ and
 $$
 f_i(u)=u^{d_i}+b_{i,d_i-1}u^{d_i-1}+\cdots+b_{i,0}
 $$
 is the minimal polynomial of $\alpha_i$ over $\mathfrak k(\alpha_1,\ldots,\alpha_{i-1})$.
 
\end{enumerate}

The algorithm terminates with $\Omega<\infty$ if and only if either
\begin{equation}\label{eqL15}
\overline n_{\Omega}=[G(\nu(P_0),\ldots, \nu(P_{\Omega})):G(\nu(P_0),\ldots, \nu(P_{\Omega-1}))]=\infty
\end{equation}
or 
\begin{equation}\label{eqL10}
\begin{array}{l}
\mbox{$\overline n_{\Omega}<\infty$ (so that $\alpha_{\Omega}$ is defined as in 4)) and}\\
\mbox{$d_{\Omega}=[\mathfrak k(\alpha_1,\ldots,\alpha_{\Omega}):\mathfrak k(\alpha_1,\ldots,\alpha_{\Omega-1})]=\infty$.}
\end{array}
\end{equation}
If $\overline n_{\Omega}=\infty$, set $\alpha_{\Omega}=1$.

 \end{Theorem} 
  
 \begin{proof} Consider the following statements $A(i)$, $B(i)$, $C(i)$, $D(i)$ for $1\le i<\Omega$: 
  
  $$
\begin{array}{ll}
&\mbox{There exists $U_i=P_0^{w_0(i)}P_1^{w_1(i)}\cdots P_{i-1}^{w_{i-1}(i)}$ for some $w_j(i)\in\NN$}\\
&\mbox{and 
$0\le w_j(i)<n_j$ for $1\le j\le i-1$}\\
&\mbox{such that $\overline n_i\nu(P_i)=\nu(U_i)$. Let $\alpha_i=[\frac{P_i^{\overline n_i}}{U_i}]\in L$ and}\\
A(i)&\mbox{$f_i(u)=u^{d_i}+b_{i,d_i-1}u^{d_i-1}+\cdots +b_{i,0}\in \mathfrak k(\alpha_1,\ldots,\alpha_{i-1})[u]$}\\
&\mbox{be the minimal polynomial of $\alpha_i$.}\\
&\mbox{Let $d_i$ be the degree of $f_i(u)$, and $n_i=\overline n_id_i$. Then
there exist $a_{s,t}\in CS$ }\\
&\mbox{and $j_0(s,t), j_1(s,t),\ldots, j_{i-1}(s,t)\in \NN$ with $0\le j_k(s,t)<n_k$}\\
&\mbox{for $k\ge 1$ and $0\le t<\overline d_i$ such that}\\
&$$
\nu(P_0^{j_0(s,t)}P_1^{j_1(s,t)}\cdots P_{i-1}^{j_{i-1}(s,t)}P_{i}^{t\overline n_i})=\overline n_id_i\nu(P_i)
$$\\
&\mbox{for all $s,t$  and}
\end{array}
$$
\begin{equation}\label{eqM1}
\,\,\,\,\,\,\,\,\,\,P_{i+1}:=P_i^{\overline n_id_i}
+\sum_{t=0}^{d_i-1}\left(\sum_{s=1}^{\lambda_t}a_{s,t}P_0^{j_0(s,t)}P_1^{j_1(s,t)}\cdots P_{i-1}^{j_{i-1}(s,t)}\right)P_i^{t\overline n_i}
\end{equation}
$$
\begin{array}{ll}
&\mbox{ satisfies}\\ 
&\,\,\,\,\,\,\,\,\,\,b_{i,t}=\left[\sum_{s=1}^{\lambda_t}a_{s,t}\frac{P_0^{j_0(s,t)}P_1^{j_1(s,t)}\cdots P_{i-1}^{j_{i-1}(s,t)}}{U_i^{d_i-t}}\right]\\
&\mbox{ for $0\le t\le d_i-1$. In particular,} 
\end{array}
$$
\begin{equation}\label{eqM2}
\nu(P_{i+1})>n_i\nu(P_i).
\end{equation}

\vskip .2truein
$$
\begin{array}{ll}
B(i)&\mbox{ Suppose that $M$ is a Laurent monomial in $P_0,P_1,\ldots, P_i$ and $\nu(M)=0$. Then}\\
&\mbox{ there exist $s_i\in\ZZ$ such that }\\
&\,\,\,\,\,\,\,\,\,\, M=\prod_{j=1}^i\left[\frac{P_j^{\overline n_j}}{U_j}\right]^{s_j},\\
&\mbox{ so that}\\
&\,\,\,\,\,\,\,\,\,\, [M]\in \mathfrak k(\alpha_1,\ldots,\alpha_i).
\end{array}
$$
\vskip .2truein
$$
\begin{array}{ll}
&\mbox{ Suppose that $\lambda\in \mathfrak k(\alpha_1,\ldots,\alpha_i)$ and $N$ is a Laurent monomial}\\
&\mbox{ in $P_0,P_1,\ldots,P_i$
such that $\gamma=\nu(N)\ge n_i\nu(P_i)$. Then there exists}\\
C(i)&\,\,\,\,\,\,\,\,\,\,
G=\sum_jc_jP_0^{\tau_0(j)}P_1^{\tau_1(j)}\cdots P_i^{\tau_i(j)}\\
&\mbox{ with $\tau_0(j),\ldots, \tau_i(j)\in \NN$, $0\le \tau_k(j)<n_k$ for $1\le k\le i$ and $c_j\in CS$ such that}\\
&\,\,\,\,\,\,\,\,\,\,\nu(P_0^{\tau_0(j)}P_1^{\tau_1(j)}\cdots P_i^{\tau_i(j)})=\gamma\mbox{ for all $j$}\\
&\mbox{ and }\\
&\,\,\,\,\,\,\,\,\,\, \left[\frac{G}{N}\right]=\lambda.
\end{array}
$$
\vskip .2truein
$$
\begin{array}{ll}
&\mbox{ Suppose that  $m\in \ZZ_+$, $j_k(l)\in\NN$  for $1\le l\le m$ and $0\le j_k(l)<n_k$}\\
&\mbox{ for $1\le k\le i$ are such that the $(j_0(l),j_1(l),\ldots,j_i(l))$ are distinct for $1\le l\le m$, and }\\
 &
 \,\,\,\,\,\,\,\,\,\,\nu(P_0^{j_0(l)}P_1^{j_1(l)}\cdots P_i^{j_i(l)})=\nu(P_0^{j_0(1)}\cdots P_i^{j_i(1)})
 \\
 
 D(i)&\mbox{ for $1\le l\le m$.
 Then}\\
 &
 \,\,\,\,\,\,\,\,\,\,1,\left[\frac{P_0^{j_0(2)}P_1^{j_1(2)}\cdots P_i^{j_i(2)}}{P_0^{j_0(1)}P_1^{j_1(1)}\cdots P_i^{j_i(1)}}\right],
 \ldots,
 \left[\frac{P_0^{j_0(m)}P_1^{j_1(m)}\cdots P_i^{j_i(m)}}{P_0^{j_0(1)}P_1^{j_1(1)}\cdots P_i^{j_i(1)}}\right]
 \\
 &\mbox{are linearly independent over $\mathfrak k$.}
 \end{array}
 $$

We will leave the proofs of $A(1)$, $B(1)$, $C(1)$ and $D(1)$ to the reader, as they are  an easier variation of the following inductive statement, which we will prove. 

Assume that $i\ge 1$ and $A(i)$, $B(i)$, $C(i)$ and $D(i)$ are true. We will prove that $A(i+1)$, $B(i+1)$ and $C(i+1)$
and $D(i+1)$  are true. Let $\beta_j=\nu(P_j)$ for $0\le j\le i+1$.
By Lemma \ref{Lemma2}, there exists 
$U_{i+1}=P_0^{w_0(i)}P_1^{w_1(i)}\cdots P_{i}^{w_{i}(i)}$ for some $w_j(i)\in\NN$
such that  
$0\le w_j(i)<n_j$ for $1\le j\le i$ and 
$\nu(U_{i+1})=\overline n_{i+1}\beta_{i+1}$ (where $\overline n_{i+1} =[G(\beta_0,\ldots,\beta_{i+1}):G(\beta_0,\ldots, \beta_i)]$).

Let $f_{i+1}(u)$ be the minimal polynomial of 
$$
\alpha_{i+1}=\left[\frac{P_{i+1}^{\overline n_{i+1}}}{U_{i+1}}\right]
$$
 over $\mathfrak k(\alpha_1,\ldots,\alpha_i)$. Let $d=d_{i+1}=\deg f_{i+1}$.
Expand
$$
f_{i+1}(u)=u^d+b_{d-1}u^{d-1}+\cdots+b_0
$$
with $b_j\in \mathfrak k(\alpha_1,\ldots,\alpha_i)$. For $j\ge 1$,
$$
\nu(U_{i+1}^j)=j\overline n_{i+1}\beta_{i+1}\ge \beta_{i+1}>n_i\beta_i.
$$
In the inductive statement $C(i)$, take $N=U_{i+1}^{d-t}$ for $0\le t< d=d_{i+1}$, to obtain for $0\le t< d_{i+1}$,
\begin{equation}\label{eqF3}
G_t=\sum_{s=1}^{\lambda_t}a_{s,t}P_0^{j_0(s,t)}P_1^{j_1(s,t)}\cdots P_{i}^{j_{i}(s,t)}
\end{equation}
with $a_{s,t}\in CS$, $j_k(s,t)\in\NN$ and $0\le j_k(s,t)<n_k$ for $1\le k\le i$ such that
$$
\nu(G_t)=\nu(P_0^{j_0(s,t)}P_1^{j_1(s,t)}\cdots P_{i}^{j_{i}(s,t)})=(d-t)\overline n_{i+1}\beta_{i+1}
$$
for all $s,t$ and
\begin{equation}\label{eqF5}
\left[\frac{G_t}{U_{i+1}^{d-t}}\right]=b_t.
\end{equation}

Set
\begin{equation}\label{eqF1}
\begin{array}{lll}
P_{i+2}&=&P_{i+1}^{\overline n_{i+1}d_{i+1}}+G_{d-1}P_{i+1}^{\overline n_{i+1}(d_{i+1}-1)}+\cdots+G_{0}\\
&=& P_{i+1}^{\overline n_{i+1}d_{i+1}}+\sum_{t=0}^{d-1}
\sum_{s=1}^{\lambda_t}a_{s,t}P_0^{j_0(s,t)}P_1^{j_1(s,t)}\cdots P_{i}^{j_{i}(s,t)}P_{i+1}^{t\overline n_{i+1}}.
\end{array}
\end{equation}
We have established $A(i+1)$.
\vskip .2truein

Suppose $M$ is a Laurent polynomial in $P_0,P_1,\ldots, P_{i+1}$ and $\nu(M)=0$. We have a factorization
$$
M=P_0^{a_0}P_1^{a_1}\cdots P_i^{a_i}P_{i+1}^{a_{i+1}}
$$
with all $a_j\in\ZZ$. Thus $a_{i+1}\beta_{i+1}\in G(\beta_0,\ldots,\beta_i)$, so that $\overline n_{i+1}$ divides $a_{i+1}$.
Let $s=\frac{a_{i+1}}{\overline n_{i+1}}$. Then 
$$
M=U_{i+1}^s(P_0^{a_0}P_1^{a_1}\cdots P_i^{a_i})\left(\frac{P_{i+1}^{\overline n_{i+1}}}{U_{i+1}}\right)^s.
$$
Now $U_{i+1}^sP_0^{a_0}\cdots P_i^{a_i}$ is a Laurent monomial in $P_0,\ldots, P_i$ of value zero, so the validity of $B(i+1)$
follows from the inductive assumption $B(i)$. 
\vskip .2truein
We now establish $C(i+1)$. Suppose $\lambda\in \mathfrak k(\alpha_1,\ldots,\alpha_{i+1})$ and $N$ is a Laurent monomial in $P_0,P_1,\ldots, P_{i+1}$ such that $\gamma=\nu(N)\ge n_{i+1}\nu(P_{i+1})$. We have
$$
\gamma\ge n_{i+1}\beta_{i+1}=\overline n_{i+1}d_{i+1}\beta_{i+1}\ge \overline  n_{i+1}\beta_{i+1}.
$$
By Lemma \ref{Lemma2} there exist $r_0,r_1,\ldots,r_i,k\in\NN$  such that $0\le r_j<n_j$ for $1\le j\le i$ and $0\le k<\overline n_{i+1}$ such that 
$$
\overline N=P_0^{r_0}P_1^{r_1}\cdots P_i^{r_i}P_{i+1}^k
$$
satisfies $\nu(\overline N)=\gamma$. Let $\tilde N=P_0^{r_0}P_1^{r_1}\cdots P_i^{r_i}$, so that $\overline N=\tilde NP_{i+1}^k$.
Let $\tau=[\frac{N}{\overline N}]$. We have that $0\ne \tau\in \mathfrak k(\alpha_1,\ldots,\alpha_{i+1})$ by $B(i+1)$.

Suppose $0\le j\le d_{i+1}-1$. Then
\begin{equation}\label{eqF2}
\begin{array}{lll}
\nu\left(\frac{\tilde N}{U_{i+1}^j}\right)&=&\nu(\tilde N)-j\nu(U_{i+1})\\
&\ge& \gamma-(\overline n_{i+1}-1)\beta_{i+1}-(d_{i+1}-1)\overline n_{i+1}\beta_{i+1}\\
&\ge & \overline n_{i+1}d_{i+1}\beta_{i+1}-\overline n_{i+1}\beta_{i+1}+\beta_{i+1}
-d_{i+1}\overline n_{i+1}\beta_{i+1}+\overline n_{i+1}\beta_{i+1}\\
&\ge& \beta_{i+1}>n_i\beta_i.
\end{array}
\end{equation}
Write
$$
\tau\lambda=e_0+e_1\alpha_{i+1}+\cdots+e_{d_{i+1}-1}\alpha_{i+1}^{d_{i+1}-1}
$$
with $e_j\in \mathfrak k(\alpha_1,\ldots,\alpha_i)$. By the inductive statement $C(i)$ and (\ref{eqF2}), there exist for $0\le j\le d_{i+1}-1$
$$
H_j=\sum_kc_{k,j}P_0^{\delta_0(k,j)}P_1^{\delta_1(k,j)}\cdots P_i^{\delta_i(k,j)}
$$
with $\delta_0(k,j),\delta_1(k,j),\ldots,\delta_i(k,j)\in \NN$, $0\le \delta_l(k,j)<n_l$ for $1\le l$ and $c_{k,j}\in CS$
for all $k,j$  such that
$$
\nu(P_0^{\delta_0(k,j)}P_1^{\delta_1(k,j)}\cdots P_i^{\delta_i(k,j)})=\nu\left(\frac{\tilde N}{U_{i+1}^j}\right)
$$
for all $j,k$ and
$$
\left[\frac{H_j}{\left(\frac{\tilde N}{U_{i+1}^j}\right)}\right]=e_j
$$
for all $j$.
Set
$$
G= H_0P_{i+1}^k+H_1P_{i+1}^{\overline n_{i+1}+k}+\cdots +H_{d_{i+1}-1}P_{i+1}^{\overline n_{i+1}(d_{i+1}-1)+k}.
$$
We have
$$
\overline n_{i+1}(d_{i+1}-1)+k<\overline n_{i+1}(d_{i+1}-1)+\overline n_{i+1}\le \overline n_{i+1}d_{i+1}=n_{i+1},
$$
and
$$
\frac{G}{\overline N}=\frac{H_0}{\tilde N}+\left(\frac{H_1U_{i+1}}{\tilde N}\right)\left(\frac{P_{i+1}^{\overline n_{i+1}}}{U_{i+1}}\right)
+\cdots+ \left(\frac{H_{d_{i+1}-1}U_{i+1}^{d_{i+1}-1}}{\tilde N}\right)\left(\frac{P_{i+1}^{\overline n_{i+1}}}{U_{i+1}}\right)^{d_{i+1}-1}.
$$
We have
$$
\left[\frac{G}{\overline N}\right]=e_0+e_1\alpha_{i+1}+\cdots+e_{d_{i+1}-1}\alpha_{i+1}^{d_{i+1}-1}=\tau\lambda.
$$
Thus
$$
\left[\frac{G}{N}\right] =\left[\frac{G}{\overline N}\right]\left[\frac{\overline N}{N}\right]=\tau\lambda\tau^{-1}=\lambda.
$$
We have established $C(i+1)$.
\vskip .2truein

Suppose that $D(i+1)$ is not true. We will obtain a contradiction. Under the assumption that $D(i+1)$ is not true, 
there exists 
 $m\in \ZZ_+$, $j_k(l)\in\NN$  for $1\le l\le m$ with $0\le j_k(l)<n_k$
for $1\le k\le i+1$  such that  $(j_0(l),j_1(l),\ldots,j_{i+1}(l))$ are distinct for $1\le l\le m$, and 
 $$
 \nu(P_0^{j_0(l)}P_1^{j_1(l)}\cdots P_{i+1}^{j_{i+1}(l)})=\nu(P_0^{j_0(1)}P_1^{j_0(1)}\cdots P_{i+1}^{j_{i+1}(1)})
 $$
  for $1\le l\le m$
  and $\tilde a_{l}\in \mathfrak k$ for $1\le l\le m$ not all zero such that
 $$
 \tilde a_1+\tilde a_2\left[\frac{P_0^{j_0(2)}P_1^{j_1(2)}\cdots P_{i+1}^{j_{i+1}(2)}}{P_0^{j_0(1)}P_1^{j_1(1)}\cdots P_{i+1}^{j_{i+1}(1)}}\right]+
 \cdots+
 \tilde a_m\left[\frac{P_0^{j_0(m)}P_1^{j_1(m)}\cdots P_i^{j_i(m)}}{P_0^{j_0(1)}P_1^{j_1(1)}\cdots P_{i+1}^{j_{i+1}(1)}}\right]=0.
 $$
 
 $(j_{i+1}(l)-j_{i+1}(1))\beta_{i+1}\in G(\beta_0,\ldots,\beta_i)$ for $1\le l\le m$, so $\overline n_{i+1}$ divides $(j_{i+1}(l)-j_{i+1}(1))$ for all $l$. Thus after possibly dividing all monomials $P_0^{j_0(l)}P_1^{j_1(l)}\cdots P_{i+1}^{j_{i+1}(l)}$ by a common power of $P_{i+1}$, we may assume that 
 \begin{equation}\label{eqH1}
 \overline n_{i+1}\mbox{ divides }j_{i+1}(l)\mbox{ for all $l$.} 
 \end{equation}
 After possibly reindexing the $P_0^{j_0(l)}P_1^{j_1(l)}\cdots P_{i+1}^{j_{i+1}(l)}$, we may assume that $j_{i+1}(1)=\overline n_{i+1}\phi$ is the largest value of $j_{i+1}(l)$.

For $1\le l\le m$, define $a_l\in CS$ by $\overline a_l=\tilde a_l$. Let
$$
Q=\sum_{l=1}^ma_lP_0^{j_0(l)}P_1^{j_1(l)}\cdots P_{i+1}^{j_{i+1}(l)}.
$$
 Let 
$$ 
 Q_s=\sum_{j_{i+1}(l)=s\overline n_i}a_lP_0^{j_0(l)}P_1^{j_1(l)}\cdots P_i^{j_i(l)}
 $$
 for $0\le s\le \phi$. Then
 \begin{equation}\label{eqH2}
 Q=\sum_{s=0}^{\phi}Q_sP_{i+1}^{\overline n_{i+1}s}.
 \end{equation}
 Let
 $$
 c_s=\left[\frac{Q_s}{P_0^{j_0(1)}P_1^{j_1(1)}\cdots P_i^{j_i(1)}U_{i+1}^{(\phi-s)}}\right]\in \mathfrak k(\alpha_1,\ldots\alpha_i)
 $$
 by $B(i)$. We further have that $c_{\phi}\ne 0$ by $D(i)$ since the monomials are all distinct.
 
 Dividing $Q$ by $P_0^{j_0(1)}P_1^{j_1(1)}\cdots P_i^{j_i(1)}U_{i+1}^{\phi}$, we have
 $$
 0=\sum_{s=0}^{\phi}c_s\alpha_{i+1}^s.
 $$
 Thus the minimal polynomial $f_{i+1}(u)$ of $\alpha_{i+1}$ divides $g(u)=\sum_{s=0}^\phi c_su^s$ in $\mathfrak k(\alpha_1,\ldots,\alpha_i)[u]$. But then $\phi\ge d_{i+1}$, so that $j_{i+1}(1)=\overline n_{i+1}\phi\ge n_{i+1}$, a contradiction.
 
\end{proof}

 \begin{Remark}\label{RemarkH1}  Theorem \ref{Theorem1*} can be stated without recourse to a coefficient set.
 To give this statement (which has the same proof) (\ref{eq12*}) must be replaced with ``$c_k$ are units in $R$ for $1\le k\le \lambda_i$''. In the proof, the statement ``$a_{s,t}\in CS$'' in $A(i)$ must be replaced with ``$a_{s,t}$ units in $R$ or $a_{s,t}=0$''.
 The statement ``$c_j\in CS$'' in $C(i)$ must be replaced with ``$c_j$ is a unit in $R$ or $c_j=0$''.
 \end{Remark}
 
 \begin{Remark}\label{RemarkH10} For $i>0$, there is an expression
 $$
 P_{i+1}=y^{n_1\cdots n_i}+x\Theta_{i+1}
 $$
 with $\Theta_{i+1}\in R$. This follows by considering the expression (\ref{eq11*}) and the various constraints on the values of the terms of the monomials in this expression.
 \end{Remark}

\begin{Remark}\label{RemarkH2} The algorithm of Theorem \ref{Theorem1*} concludes with $\Omega<\infty$ if and only if
$\nu(P_{\Omega})\not\in \QQ\nu(x)$ (so that $\mbox{rank}(\nu)=2$) or $\nu$ is discrete of rank 1 with 
$\mbox{trdeg}_{R/\mathfrak m_R}V_{\nu}/\mathfrak m_{\nu}=1$ (so that $\nu$ is  divisorial).
\end{Remark}

\begin{proof} From  Theorem \ref{Theorem1*}, we see that the algorithm terminates with $\Omega<\infty$ if and
only if either
$$
[G(\nu(P_0),\ldots,\nu(P_{\Omega})):G(\nu(P_0),\ldots,\nu(P_{\Omega-1}))]=\infty
$$
or
$$
[G(\nu(P_0),\ldots,\nu(P_{\Omega})):G(\nu(P_0),\ldots,\nu(P_{\Omega-1}))]<\infty\mbox{ and }[\mathfrak k(\alpha_1,\ldots,\alpha_{\Omega}):\mathfrak k(\alpha_1,\ldots,\alpha_{\Omega-1})]=\infty.
$$
\end{proof}

\begin{Remark}\label{Remark2} Suppose that $\Omega=\infty$ and $n_i=1$ for $i\gg 0$
in the conclusions of Theorem \ref{Theorem1*}. Then $\nu$ is discrete, and $V_{\nu}/\mathfrak m_{\nu}$ is finite over $\mathfrak k$.
\end{Remark}

\begin{proof} We first deduce a consequence of the assumption that $\Omega=\infty$ and $n_i=1$ for $i\gg 0$.
There exists $i_0\in\ZZ_+$ such that $n_i=1$
for all $i\ge i_0$. Thus for $i \ge i_0$,  $P_{i+1}$ is the sum of $P_{i}$ and a $\mathfrak k$-linear combination of monomials $M$ in $x$
and the finitely many $P_j$ with $j<i_0$, and with $\nu(M)=\nu(P_i)$.  We see that the $P_i$ form a Cauchy sequence in $\hat R$ 
 whose limit $f$ in $\hat R$ is nonzero (by Remark \ref{RemarkH10}), and such that 
$\mbox{lim}_{i \rightarrow \infty}\nu(P_i)=\infty$. 

Thus $I_{\hat R}\ne (0)$, $\nu$ is discrete and $V_{\nu}/\mathfrak m_{\nu}$ is finite over $\mathfrak k$ by the proof of Proposition \ref{Prop17}.

\end{proof}

\begin{Remark}\label{RemarkH5} 
Suppose that  $V_{\nu}/\mathfrak m_{\nu}= R/\mathfrak m_R$ in the hypotheses of Theorem \ref{Theorem1*} (so that there
is no residue field extension). Then the $P_i$ constructed by the algorithm are binomials for $i\ge 2$; (\ref{eq11*}) becomes
$$
P_{i+1}=P_i^{\overline n_i}+cU_i=P_i^{\overline n_i}+cP_0^{w_0(i)}\cdots P_{i-1}^{w_{i-1}(i)}
$$
for some $0\ne c\in CS$.
\end{Remark}

\begin{Example} There exists a rank 2 valuation $\nu$ dominating $R=\mathfrak k[x,y]_{(x,y)}$ such that the set
$$
\{\nu(P_0),\nu(P_1),\nu(P_2),\ldots\}
$$
 does not generate the semigroup $S^{R}(\nu)$.
\end{Example}

\begin{proof} Suppose that $\mathfrak k$ is a field  of characteristic zero.
We define a rank 2 valuation $\hat\nu$ on $\mathfrak k [[x,y]]$. Let $g(x,y)=y-x\sqrt{x+1}$. For $0\ne f(x,y)\in \mathfrak k [[x,y]]$, we have a factorization $f=g^nh$ where $n\in\NN$ and $g\not\,\mid h$. The rule
$$
\hat \nu(f)=(n,\mbox{ord}(h(x,x\sqrt{1+x})))\in(\ZZ^2)_{\rm lex}
$$
then defines a rank 2 valuation dominating  $\mathfrak k[[x,y]]$ with value group $(\ZZ^2)_{\rm lex}$.

We have that $(g)\cap \mathfrak k [x,y]=(y^2-x^2-x^3)$. Thus $\hat\nu$ restricts to a rank 2 valuation $\nu$  which  dominates the maximal ideal $\mathfrak n=(x,y)$ of $\mathfrak k[x,y]$. 
Expand
 $$
 x\sqrt{1+x}=\sum_{j\ge 1}a_ix^j=x+\frac{1}{2}x^2-\frac{1}{8}x^3+\cdots
$$
as a series with all $a_j\in \mathfrak k$ non zero.
Applying the algorithm of Theorem \ref{Theorem1*}, we construct the infinite sequence of polynomials 
$P_1,P_2,\cdots$ where $P_0=x$,
$P_1=y$ and $P_i=y-\sum_{j=1}^{i-1}a_ix^i$ for $i\ge 2$. 
We have that $\nu(P_i)=(0,i)$ for $i\ge 0$. 
However, $\nu(y^2-x^2-x^3)=(1,1)$.

Thus the set $\{\nu(x),\nu(P_1),\nu(P_2),\ldots\}$ does not generate the semigroup $S^{R}(\nu)$.

\end{proof}

\begin{Lemma}\label{LemmaG10} 
Suppose that $\nu$ is a valuation dominating $R$.
Let 
$$
P_0=x, P_1=y, P_2,\ldots
$$
be the sequence of elements of $R$ constructed by  Theorem \ref{Theorem1*}. Set
$\beta_i=\nu(P_i)$ for $i\ge 0$. 
Suppose that $P_0^{m_0}P_1^{m_1}\cdots P_r^{m_r}$ is a monomial in $P_0,\ldots,P_r$ and $m_i\ge n_i$ for some $i\ge 1$. Let $\rho=\nu(P_0^{m_0}P_1^{m_1}\cdots P_r^{m_r})$. Then with the notation of (\ref{eqM1}),
\begin{equation}\label{eqW1}
\begin{array}{lll}
P_0^{m_0}\cdots P_r^{m_r}&=&
-\sum_{t=0}^{d_i-1}\sum_{s=1}^{\lambda_t}a_{s,t}P_0^{m_0+j_0(s,t)}\cdots P_{i-1}^{m_{i-1}+j_{i-1}(s,t)}
P_i^{m_i-n_i+t\overline n_i}P_{i+1}^{m_{i+1}}\cdots P_r^{m_r}\\
&&+P_0^{m_0}\cdots P_i^{m_i-n_i}P_{i+1}^{m_{i+1}+1}\cdots P_r^{m_r}.
\end{array}
\end{equation}
All terms in the first sum of (\ref{eqW1}) have value $\rho$ and $\nu(P_0^{m_0}\cdots P_i^{m_i-n_i}P_{i+1}^{m_{i+1}+1}\cdots P_r^{m_r})>\rho$.

Suppose that $W$ is a Laurent monomial in $P_0,\ldots, P_r$ such that $\nu(W)=\rho$. Then
\begin{equation}\label{eqW2}
\left[\frac{P_0^{m_0}P_1^{m_1}\cdots P_r^{m_r}}{W}\right]
=-\sum_{t=0}^{d_i-1}\sum_{s=1}^{\lambda_t}\overline a_{s,t}\left[
\frac{P_0^{m_0+j_0(s,t)}\cdots P_{i-1}^{m_{i-1}+j_{i-1}(s,t)}P_i^{m_i-n_i+\overline n_it}P_{i+1}^{m_{i+1}}\cdots P_r^{m_r}}{W}\right]
\end{equation}
and
\begin{equation}\label{eqW3}
\begin{array}{l}
(m_0+j_0(s,t))+\cdots+(m_{i-1}+j_{i-1}(s,t))+(m_i-n_i+t\overline n_i)+m_{i+1}+\cdots+m_r\\
>m_0+m_1+\cdots+m_r
\end{array}
\end{equation}
for all terms in the first sum of (\ref{eqW1}).
\end{Lemma}

\begin{proof} We have
$$
P_0^{m_0}\cdots P_r^{m_r}=P_0^{m_0}\cdots P_i^{n_i}P_i^{m_i-n_i}\cdots P_r^{m_r}
$$
where $m_i-n_i\ge 0$. Substituting (\ref{eqM1}) for $P_i^{n_i}$, we obtain
equation (\ref{eqW1}).
We compute, from the first term of (\ref{eqW1}),
$$
\begin{array}{l}
-\sum_{t=1}^{d_i-1}\sum_{s=1}^{\lambda_t}\overline a_{s,t}\left[
\frac{P_0^{m_0+j_0(s,t)}\cdots P_r^{m_r}}{W}\right]\\
=-\left[\frac{P_0^{m_0}\cdots P_i^{m_i-n_i}\cdots P_r^{m_r}U_i^{d_i}}{W}\right]
\left(\sum_{t=0}^{d_i-1}\sum_{s=1}^{\lambda_t}\overline a_{s,t}\left[\frac{P_0^{j_0(s,t)}\cdots P_{i-1}^{j_{i-1}(s,t)}}
{U_i^{d_i-t}}\right]\left[\frac{P_i^{\overline n_i}}{U_i}\right]^t\right)\\
\\
=-\left[\frac{P_0^{m_0}\cdots P_i^{m_i-n_i}\cdots P_r^{m_r}U_i^{d_i}}{W}\right]\left(\sum_{t=0}^{d_i-1}b_{i,t}\alpha_i^t\right)\\
=\left[\frac{P_0^{m_0}\cdots P_i^{m_i-n_i}\cdots P_r^{m_r}U_i^{d_i}}{W}\right]\alpha_i^{d_i}\\
=\left[\frac{P_0^{m_0}\cdots P_i^{m_i-n_i}\cdots P_r^{m_r}U_i^{d_i}}{W}\right]\left[\frac{P_i^{\overline n_i}}{U_i}\right]^{d_i}\\
=\left[\frac{P_0^{m_0}\cdots P_i^{m_i}\cdots P_r^{m_r}}{W}\right],
\end{array}
$$
giving (\ref{eqW2}).  For all $s,t$ (with $0\le t\le d_i-1$),
$$
\begin{array}{lll}
n_i\beta_i&=&j_0(s,t)\beta_0+j_1(s,t)\beta_1+\cdots+j_{i-1}(s,t)\beta_{i-1}+\overline n_it\beta_i\\
&<& \left( j_0(s,t)+j_1(s,t)+\cdots +j_{i-1}(s,t)+\overline n_it\right)\beta_i
\end{array}
$$
so
$$
n_i<j_0(s,t)+j_1(s,t)+\cdots +j_{i-1}(s,t)+\overline n_it.
$$
(\ref{eqW3}) follows.

\end{proof}

\begin{Theorem}\label{TheoremG2} 
Suppose that $\nu$ is a valuation dominating $R$.
Let 
$$
P_0=x, P_1=y, P_2,\ldots
$$
be the sequence of elements of $R$ constructed by  Theorem \ref{Theorem1*}. Set
$\beta_i=\nu(P_i)$ for $i\ge 0$. Suppose that $f\in R$  and there exists $n\in\ZZ_+$ such that $\nu(f)<n\nu(\mathfrak m_R)$. 
Then there exists an expansion 
$$
f=\sum_{I}a_IP_0^{i_0}P_1^{i_1}\cdots P_r^{i_r}+\sum_J\phi_JP_0^{j_0}\cdots P_r^{j_r}+h
$$
where $r\in\NN$, $a_{I}\in CS$, $I,J\in \NN^{r+1}$, $\nu(P_0^{i_0}P_1^{i_1}\cdots P_r^{i_r})=\nu(f)$ for all $I$ in the first sum,  $0\le i_k<n_k$ for $1\le k\le r$, $\nu(P_0^{j_0}\cdots P_r^{j_r})>\nu(f)$ for all terms in the second sum, $\phi_J\in R$ and $h\in \mathfrak m_R^n$.

The first sum is uniquely determined by these conditions.
\end{Theorem}

\begin{proof} We first prove existence. 
We  have an expansion
$$
f=\sum a_{i_0,i_1}x^{i_0}y^{i_1}+h_0
$$
with $a_{i_0,i_1}\in CS$ and $h_0\in \mathfrak m_R^n$. 
More generally, suppose that we have an expansion
\begin{equation}\label{eqG5}
f=\sum a_IP_0^{i_0}P_1^{i_1}\cdots P_r^{i_r}+h
\end{equation}
for some $r\in \ZZ_+$, $I=(i_0,\ldots,i_r)\in\NN^{r+1}$, $a_I\in CS$ and $h\in \mathfrak m_R^n$.
Let 
$$
\rho=\mbox{min}\{\nu(P_0^{i_0}P_1^{i_1}\cdots P_r^{i_r})\mid a_I\ne 0\}.
$$
We can rewrite (\ref{eqG5}) as 
\begin{equation}\label{eqG6}
f=\sum_J a_JP_0^{j_0}P_1^{j_1}\cdots P_r^{j_r}+\sum_{J'}a_{J'}P_0^{j_0'}P_1^{j_1'}\cdots P_r^{j_r'} +h
\end{equation}
where the terms in the first  sum have minimal value $\nu(P_0^{j_0}P_1^{j_1}\cdots P_r^{j_r})=\rho$ and the nonzero terms in the second 
sum have value $\nu(P_0^{j_0'}P_1^{j_1'}\cdots P_r^{j_r'})>\rho$.

 If we have that the first sum is nonzero   and $0\le j_k<n_k$ for $1\le k\le r$ for all terms in the first sum of (\ref{eqG6}) then $\rho =\nu(f)$ and we have achieved the conclusions of the theorem. So suppose that one of these conditions fails.

First suppose that $\sum_Ja_JP_0^{j_0}\cdots P_r^{j_r}\ne 0$ and for some $J$, $j_i\ge n_i$ for some $i\ge 1$.
Let 
$$
a=\min\{j_0+\cdots+j_r\mid j_i\ge n_i\mbox{ for some }i\ge 1\}
$$
and let $b$ be the numbers of terms in  $\sum_Ja_JP_0^{j_0}\cdots P_r^{j_r}$ such that $j_i\ge n_i$ for some $i\ge 1$
and $j_0+\cdots+j_r=a$. Let $\sigma=(a,b)\in(\ZZ^2)_{\rm lex}$. Let $J_0=(\overline j_0,\ldots,\overline j_r)$ be such that
$a_{J_0}\ne 0$ and $\overline j_0+\cdots+\overline j_r=a$.
Write
$$
P_0^{\overline j_0}\cdots P_r^{\overline j_r}=P_0^{\overline j_0}\cdots P_i^{\overline j_i-n_i}P_i^{n_i}\cdots P_r^{\overline j_r}
$$
and substitute (\ref{eqM1}) for $P_i^{n_i}$, to obtain an expression of the form (\ref{eqW1}) of Lemma \ref{LemmaG10}.
Substitute this expression (\ref{eqW1}) for $P_0^{\overline j_0}\cdots P_r^{\overline j_r}$ in (\ref{eqG6}) and apply Remark \ref{RemarkG4},
to obtain an expression of the form (\ref{eqG6}) such that either the first sum is zero or the first sum is nonzero and all
terms in the first sum satisfy $j_i<n_i$ for $1\le i$ so that $\nu(f)=\rho$ and we have achieved the conclusions of the theorem,
or the first sum has a nonzero term which satisfies $j_i\ge n_i$ for some $i\ge 1$. By (\ref{eqW3}), we have an increase in $\sigma$ if this last case holds. 

Since there are only finitely many monomials $M$ in $P_0,\ldots P_r$ which have the value $\rho$, after a finite number of iterations of this step we must either find an expression (\ref{eqG6}) where the first sum is zero, or attain an expression
(\ref{eqG6}) satisfying the conclusions of the theorem.

If we obtain an expression (\ref{eqG6}) where the first sum is zero, then we have an expression (\ref{eqG5}) with an increase in 
$\rho$ (and possibly an increase in $r$), and we repeat the last step, either attaining the conclusions of the theorem or obtaining another increase in $\rho$.
Since there are only a finite number of monomials in the $\{P_i\}$ which have value $\le \nu(f)$, we must achieve the conclusions
of the theorem in a finite number of steps.

  Uniqueness of the first sum follows from 2) of Theorem \ref{Theorem1*}.

\end{proof}

\begin{Theorem}\label{Corollary1*} Suppose that $\nu$ is a rank 1 valuation which dominates $R$ and $\nu(x)=\nu(\mathfrak m_R)$. Then
\begin{enumerate} 
\item[a)] The  set $\{{\rm in}_{\nu}(x)\}\cup\{ {\rm in}_{\nu}(P_i)\mid n_i>1\}$
minimally generates ${\rm gr}_{\nu}(R)$ as a $\mathfrak k$-algebra. 
\item[b)] The set 
$$
\{\nu(x)\}\cup \{ \nu(P_i)\mid \overline n_i>1\}
$$ 
minimally generates the semigroup $S^{R}(\nu)$.
\item[c)] $V_{\nu}/\mathfrak m_{\nu}=\mathfrak k(\alpha_i\mid d_i>1)$ where $\alpha_i$ is defined by 4) (and possibly (\ref{eqL10})) of Theorem \ref{Theorem1*}.
\end{enumerate}
\end{Theorem}

\begin{proof} Theorem \ref{TheoremG2} implies that the set $\{{\rm in}_{\nu}(x)\}\cup\{ {\rm in}_{\nu}(P_i)\mid n_i>1\}$ generates $\mbox{gr}_{\nu}(R)$ as a $\mathfrak k$-algebra. 
We will show that the set generates $\mbox{gr}_{\nu}(R)$ minimally.
Suppose that it doesn't. Then there exists an $i\in \NN$ such that $n_i>1$ if $i>0$ and a sum
\begin{equation}\label{eqZ30}
H=\sum_Jc_JP_0^{j_0}\cdots P_r^{j_r}
\end{equation}
for some $r\in\NN$ with $c_J\in CS$ such that the monomials $P_0^{j_0}\cdots P_r^{j_r}$  have value $\nu(P_0^{j_0}\cdots P_r^{j_r})=\nu(P_i)$ 
with $j_i=0$ and $j_k=0$ if $n_k=1$  for $1\le k\le r$ for all $J$, and  
$$
\nu(\sum_Jc_JP_0^{j_0}\cdots P_r^{j_r}-P_i)>\nu(P_i).
$$
We thus have by 1) of Theorem \ref{Theorem1*} and since $\nu(P_0)=\nu(\mathfrak m_R)$, 
that $r\le i-1$. Thus $i\ge 1$.
By Theorem \ref{TheoremG2} applied to $H$, we have an expression
\begin{equation}\label{eqZ31}
P_i=\sum_Kd_K\P_0^{k_0}\cdots P_s^{k_s}+f
\end{equation}
where $s\in\NN$, $d_K\in CS$, $0\le k_l<n_l$ for $1\le l$, some $d_K\ne 0$,  $f\in R$ is such that $\nu(f)>\nu(P_i)$, and
$$
\nu(P_0^{k_0}\cdots P_s^{k_s})=\nu(H)=\nu(P_i)
$$
for all monomials in the first sum of (\ref{eqZ31}). Since the minimal value terms of the expression of $H$ in (\ref{eqZ30}) only involve $P_0,\ldots, P_{i-1}$ and all these monomials have the same value $\rho=\nu(H)$, the algorithm of Theorem \ref{TheoremG2} ends with $s\le i-1$ in (\ref{eqZ31}). But then we obtain from (\ref{eqZ31}) a contradiction to 2) of Theorem \ref{Theorem1*}.

Now a) and 3) of Theorem \ref{Theorem1*} imply statement b).

Suppose that $\lambda\in L=V_{\nu}/\mathfrak m_{\nu}$. Then $\lambda=\left[\frac{f}{f'}\right]$ for some $f,f'\in R$ with $\nu(f)=\nu(f')$. By Theorem \ref{TheoremG2}, there exist $r\in\ZZ_+$ and expressions
$$
f=\sum_{i=1}^ma_iP_0^{\sigma_0(i)}P_1^{\sigma_1(i)}\cdots P_r^{\sigma_r(i)}+h,
$$
$$
f'=\sum_{j=1}^nb_jP_0^{\tau_0(j)}P_1^{\tau_1(j)}\cdots P_r^{\tau_r(j)}+h'
$$
with $a_i,b_j\in CS$, $0\le \sigma_k(i)<n_k$ for $1\le k$ and $0\le \tau_k(j)<n_k$ for $1\le k$, the $P_0^{\sigma_0(i)}P_1^{\sigma_1(i)}\cdots P_r^{\sigma_r(i)}$, $P_0^{\tau_0(j)}P_1^{\tau_1(j)}\cdots P_r^{\tau_r(j)}$ all have the common value 
$$
\rho:= \nu(f)=\nu(f'),
$$
$h,h'\in R$  and $\nu(h)>\rho$, $\nu(h')>\rho$.
$$
\begin{array}{lll}
\lambda&=&\left(\sum_i\overline a_i[P_0^{\sigma_0(i)-\sigma_0(1)}\cdots P_r^{\sigma_r(i)-\sigma_r(1)}]\right)
\left(\sum_j\overline b_j[P_0^{\tau_0(i)-\sigma_0(1)}\cdots P_r^{\tau_r(i)-\sigma_r(1)}]\right)^{-1}\\
&&\,\,\,\,\,\,\,\,\,\,\in \mathfrak k(\alpha_1,\ldots,\alpha_r)
\end{array}
$$
by $B(r)$ of the proof of Theorem \ref{Theorem1*}.
\end{proof}
\vskip .2truein
 If $V_{\nu}/\mathfrak m_{\nu}$ is transcendental over $\mathfrak k$ then $\Gamma_{\nu}\cong \ZZ$ by Abhyankar's inequality.
 Zariski called such a  valuation a ``prime divisor of the second kind''. By c) of Theorem \ref{Corollary1*},
 $V_{\nu}/\mathfrak m_{\nu}=\mathfrak k(\alpha_i\mid d_i>1)$. There thus exists an index $i$ such that $\mathfrak k(\alpha_1,\ldots,\alpha_{i-1})$ is algebraic over $\mathfrak k$ and $\alpha_i$ is transcendental over  $\mathfrak k(\alpha_1,\ldots,\alpha_{i-1})$. Thus $\Omega=i$ in the algorithm of Theorem \ref{Theorem1*}, since $\alpha_i$ does not
 have a minimal polynomial over $\mathfrak k(\alpha_1,\ldots,\alpha_{i-1})$.
\vskip .2truein

\begin{Theorem}\label{Corollary3*} Suppose that $\nu$ is a rank 2 valuation which dominates $R$ and $\nu(x)=\nu(\mathfrak m_R)$.
Let $I_{\nu}$ be the height one prime ideal in $V_{\nu}$.
Then  one of the following three cases hold:
\begin{enumerate}
\item[1.] $I_{\nu}\cap R=\mathfrak m_R$. Then 
\begin{enumerate}
\item[a)] the finite set
$$
\{{\rm in}_{\nu}(x)\}\cup\{{\rm in}_{\nu}(P_i)\mid n_i>1\}
$$
minimally generates ${\rm gr}_{\nu}(R)$ as an $\mathfrak k$-algebra and
\item[b)] the finite set 
$$
\{\nu(x)\}\cup\{\nu(P_i)\mid \overline n_i>1\}
$$
minimally generates the semigroup $S^{R}(\nu)$.
\item[c)] $V_{\nu}/\mathfrak m_{\nu}=\mathfrak k(\alpha_i\mid d_i>1)$.
\end{enumerate}
\item[2.] $I_{\nu}\cap R=(P_{\Omega})$ is a height one prime ideal in $R$ and
\begin{enumerate}
\item[a)] 
the finite set 
$$
\{{\rm in}_{\nu}(x)\}\cup\{{\rm in}_{\nu}(P_i)\mid n_i>1\}
$$
minimally generates ${\rm gr}_{\nu}(R)$ as a $\mathfrak k$-algebra, and
\item[b)] The finite set 
$$
\{\nu(x)\}\cup\{\nu(P_i)\mid \overline n_i>1\}
$$
minimally generates the semigroup $S^{R}(\nu)$.
\item[c)] $V_{\nu}/\mathfrak m_{\nu}=\mathfrak k(\alpha_i\mid d_i>1)$.
\end{enumerate}
\item[3.] $I_{\nu}\cap R=(g)$ is a height one prime ideal in $R$ and 
\begin{enumerate}
\item[a)] 
the finite set 
$$
\{{\rm in}_{\nu}(x)\}\cup\{{\rm in}_{\nu}(P_i)\mid n_i>1\}\cup\{{\rm in}_{\nu}(g)\}
$$
minimally generates ${\rm gr}_{\nu}(R)$ as a $\mathfrak k$-algebra, and
\item[b)] The finite set 
$$
\{\nu(x)\}\cup\{\nu(P_i)\mid \overline n_i>1\}\cup\{\nu(g)\}
$$
minimally generates the semigroup $S^{R}(\nu)$.
\item[c)] $V_{\nu}/\mathfrak m_{\nu}=\mathfrak k(\alpha_i\mid d_i>1)$.
\end{enumerate}
\end{enumerate}
\end{Theorem}

\begin{proof} Since $\nu$ has rank 2, the set $\{P_i\mid n_{i}>1\}$ is a finite set since otherwise either $\Gamma_{\nu}$ is not a finitely generated group or
$V_{\nu}/\mathfrak m_{\nu}$ is not a finitely generated field extension of $\mathfrak k$, by 3) and 4) of Theorem \ref{Theorem1*}, which is a contradiction to Abhyankar's inequality. 

The case when $I_{\nu}\cap R=\mathfrak m_R$ now follows from  Theorem \ref{TheoremG2} and 2), 3) of Theorem \ref{Theorem1*}; the proof of c) is the same as the proof of c) of Theorem \ref{Corollary1*}.

Suppose that $I_{\nu}\cap R=(g)$ is a height one prime ideal in $R$. 
Suppose that $f\in R$. Then there exists $n\in\NN$ and $u\in R$ such that
$f=g^nu$ with $u\not\in (g)$. Thus 
\begin{equation}\label{eqZ15}
\nu(f)=n\nu(g)+\nu(u).
\end{equation}

Assume that $\Omega<\infty$. Then  $\nu(P_{\Omega})\not\in \QQ\nu(\mathfrak m_R)$ by Remark \ref{RemarkH2}. Then $P_{\Omega}=gf$ for some $f\in R$. We will show that $f$ is a unit in $R$. Suppose not. Then $\nu(g)<\nu(P_{\Omega})$.
Let $t=\mbox{ord}(g)$.  There exists $c\in\ZZ_+$ such that if $j_0,j_1,\ldots, j_{\Omega-1}\in\NN$ are such that
$\nu(P_0^{j_0}P_1^{j_1}\cdots P_{\Omega-1}^{j_{\Omega-1}})\ge c\nu(\mathfrak m_R)$ then  $\mbox{ord}(P_0^{j_0}P_1^{j_1}\cdots P_{\Omega-1}^{j_{\Omega-1}})>t$. We may assume that $c$ is larger than $t$. Write
$$
g=\sum_{i,j=1}^ca_{ij}x^iy^j+\Lambda
$$
with $\Lambda\in\mathfrak m_R^c$ and $a_{ij}\in CS$. 
$g$ has an expression of the form
\begin{equation}\label{eq*} 
g=\sum_Ja_JP_0^{j_0}\cdots P_{\Omega}^{j_{\Omega}}+\sum_{J'}a_{J'}P_0^{j_0'}\cdots P_{\Omega}^{j_{\Omega'}}+h
\end{equation}
with $a_J,a_{J'}\in CS$ and $h\in \mathfrak m_R^c$, and the terms in the first sum all have a common value $\rho$, which is smaller than
the values of the terms in the second sum.

Now we draw some conclusions which must hold for an expression of the form (\ref{eq*}).
We must have that 
\begin{equation}\label{eq1}
\rho<c\nu(\mathfrak m_R),
\end{equation}
 since otherwise, by our choice of $c$ and our assumption that $\mbox{ord}(f)>0$, so that
$\mbox{ord}(P_{\Omega})>\mbox{ord}(g)=t$, we would have that the right hand side of (\ref{eq*}) has order larger than $t$, which is impossible.  In particular, we have 
\begin{equation}\label{eq2}
j_{\Omega}=0
\end{equation}
 in all terms in the first sum.

We also must have that 
\begin{equation}\label{eq**}
j_i\ge n_i  \mbox{ for some $i$ with $1\le i<\Omega$ for all terms in the first sum}.
\end{equation}
This follows since otherwise we would have $\nu(g)=\rho<c\nu(\mathfrak m_R)$, which is impossible.

We apply the algorithm of Theorem 4.9 to (\ref{eq*}), and apply a substitution of the form (20) to a monomial in the first sum.
As shown in the proof of Theorem 4.9, we must obtain an expression (\ref{eq*}) with an increase in $\rho$ after a finite number of iterations,
since (\ref{eq**}) must continue to hold. Since there are only finitely many values in the semigroup $S^R(\nu)$ between 0 and $c\nu(\mathfrak m_R)$, after finitely many iterations of the algorithm we obtain an expression (\ref{eq*}) with $\rho\ge c\nu(\mathfrak m_R)$, which is a contradiction to (\ref{eq1}). This contradiction shows that $P_{\Omega}$ is a unit times $g$, 
so we may replace $g$ with $P_{\Omega}$,
and we are in Case 2 of the conclusions of the corollary. 

If $\Omega=\infty$ then $\nu(P_i)\in \QQ\nu(\mathfrak m_R)$ for all $i$ (by Remark \ref{RemarkH2}) and we are in Case 3 of the conclusions of the corollary.

The conclusions of a) and b) of Cases 2 and 3 of the corollary now follow from applying Theorem \ref{TheoremG2} and 2), 3) of 
Theorem \ref{Theorem1*} to $u$ in (\ref{eqZ15}). 

Suppose that $\lambda\in V_{\nu}/\mathfrak m_{\nu}$. Then $\lambda=\left[\frac{f}{f'}\right]$ for some $f,f'\in R$ with $\nu(f)=\nu(f')$.
We may assume (after possibly dividing out a common factor) that $g\not\,\mid f$ and $g\not\,\mid f'$. Then the proof of c) of cases 2 and 3 proceeds as in the proof of c) of Theorem \ref{Corollary1*}.

\end{proof}

\section{Valuation semigroups and residue field extension on a two dimensional regular local ring}\label{Proof1}

In this section, we prove Theorem \ref{Theorem3*} which is stated in the introduction. Theorem \ref{Theorem3*} 
gives necessary and sufficient conditions for a semigroup and field extension to be the valuation semigroup 
and residue field of a valuation dominating a regular local ring
of dimension two.

 Suppose that $\nu$ is a valuation dominating  $R$. Let $S=S^r(\nu)$ and $L=V_{\nu}/\mathfrak m_{\nu}$. Let $x,y$ be regular parameters in $R$ such that
$\nu(x)=\nu(\mathfrak m_R)$. Set $P_0=x$ and $P_1=y$. Let $\{P_i\}$ be the sequence of elements of $R$ defined by the algorithm of Theorem \ref{Theorem1*}.  We have by Remark \ref{Remark2} and its proof, that if
$$
\Omega=\infty\mbox{ and }\mbox{ $n_i=1$ for $i\gg 0$,}
$$
then  $I_{\hat R}\ne (0)$ (where $I_{\hat R}$ is the prime ideal in $\hat R$ of Cauchy sequences in $R$ satisfying (\ref{eqZ12})). Thus $\nu$ has rank 2  since $R$ is complete, and $\nu$ must satisfy Case 3 of Theorem \ref{Corollary3*}.

Set $\sigma(0)=0$ and inductively define 
$$
\sigma(i)=\min\{j\mid j>\sigma(i-1)\mbox{ and }n_j>1\}.
$$
This defines an index set $I$ of finite or infinite cardinality $\Lambda=|I|-1\ge 1$. Suppose that either $\nu$ has rank 1 or $\nu$ has
rank 2 and one of the first two cases of Theorem \ref{Corollary3*} hold for the $P_i$. Let 
$$
\beta_i=\nu(P_{\sigma(i)})\in S^R(\nu)
$$
for $i\in I$ and 
$$
\gamma_i=\left[\frac{P_{\sigma(i)}^{\overline n_{\sigma(i)}}}{U_{\sigma(i)}}\right]\in V_{\nu}/\mathfrak m_{\nu}
$$
if $i>0$ and $\sigma(i)<\Omega$ or $\sigma(i)=\Omega$ and $\overline n_{\Omega}<\infty$. Set $\gamma_{\Lambda}=1$ if $\sigma(\Lambda)=\Omega$ and $\overline n_{\Omega}=\infty$.

By Theorem \ref{Theorem1*} and Theorem \ref{Corollary1*} or \ref{Corollary3*}, $\{\beta_i\}$ and $\{\gamma_i\}$ satisfy the conditions
1) and 2) of Theorem \ref{Theorem3*}. 

Suppose that $\nu$ has rank 2 and the third case of Theorem \ref{Corollary3*} holds for the $P_i$. Then $\Lambda<\infty$. Let $I_{\nu}\cap R=(g)$ (where $I_{\nu}$ is the height one prime ideal of $V_{\nu}$). Let $\overline \Lambda=\Lambda+1$. Define
$\beta_i=\nu(P_{\sigma(i)})$ for $i<\overline\Lambda$ and $\beta_{\overline \Lambda}=\nu(g)$. Define
$$
\gamma_i=\left[\frac{P_{\sigma(i)}^{\overline n_{\sigma(i)}}}{U_{\sigma(i)}}\right]\in V_{\nu}/\mathfrak m_{\nu}
$$
for $0<i<\overline\Lambda$ and define $\gamma_{\overline\Lambda}=1$. By Theorem \ref{Theorem1*} and Case 3 of Theorem \ref{Corollary3*}, $\{\beta_i\}$ and $\{\gamma_i\}$ satisfy conditions 1) and 2) of Theorem \ref{Theorem3*}.

Now suppose that $S$ and $L$ and the given sets $\{\beta_i\}$ and $\{\alpha_i\}$ satisfy conditions 1) and 2) of the theorem.   We will construct a valuation $\nu$ which dominates $R$ with $S^R(\nu)=S$ and $V_{\nu}/\mathfrak m_{\nu}=L$.

Let
$$
f_i(u)=u^{d_i}+b_{i,d_i-1}u^{d_i-1}+\cdots+b_{i,0}
$$
be the minimal polynomial of $\alpha_i$ over $\mathfrak k(\alpha_1,\ldots,\alpha_{i-1})$, and let $n_i=\overline n_id_i$. 

We will inductively define $P_i\in R$,  a function $\nu$ on Laurent monomials in $P_0,\ldots,
P_i$ such that 
$$
\nu(P_0^{a_0}P_1^{a_1}\cdots P_i^{a_i})=a_0\beta_0+a_1\beta_1+\cdots +a_i\beta_i
$$
for $a_0,\ldots, a_i\in\ZZ$ and monomials $U_i$ in $P_0,\ldots, P_{i-1}$, such that
$$
\nu(U_i)=\overline n_i\beta_i,
$$
a function $\mbox{res}$ on the Laurent monomials $P_0^{a_0}P_1^{a_1}\cdots P_i^{a_i}$ which satisfy $\nu(P_0^{a_0}P_1^{a_1}\cdots P_i^{a_i})=0$,
such that 
\begin{equation}\label{eqL40}
\mbox{res}\left(\frac{P_j^{\overline n_j}}{U_j}\right)=\alpha_j
\end{equation}
for $1\le j\le i$.

Let $x,y$ be regular parameters in $R$. Define $P_0=x$,  $P_1=y$,  $\beta_0=\nu(P_0)$, and $\beta_1=\nu(P_1)$.
We inductively construct the $P_i$ by the procedure of the algorithm of Theorem \ref{Theorem1*}. We must modify the
inductive statement $A(i)$ of the proof of Theorem \ref{Theorem1*} as follows:

\vskip .2truein
 $$
\begin{array}{ll}
&\mbox{ There exists $U_i=P_0^{w_0(i)}P_1^{w_1(i)}\cdots P_{i-1}^{w_{i-1}(i)}$ for some $w_j(i)\in\NN$}\\
&\mbox{ and 
$0\le w_j(i)<n_j$ for $1\le j\le i-1$}\\
\overline A(i)&\mbox{ such that $\overline n_i\nu(P_i)=\nu(U_i)$. 
There exist $a_{s,t}\in CS$ }\\
&\mbox{ and $j_0(s,t), j_1(s,t),\ldots, j_{i-1}(s,t)\in \NN$ with $0\le j_k(s,t)<n_k$}\\
&\mbox{ for $k\ge 1$ and $0\le t<\overline d_i$ such that}\\
&
\nu(P_0^{j_0(s,t)}P_1^{j_1(s,t)}\cdots P_{i-1}^{j_{i-1}(s,t)}P_{i}^{t\overline n_i})=\overline n_id_i\nu(P_i)\\
&\mbox{for all $s,t$  and}
\end{array}
$$
\begin{equation}\label{eqM1*}
\,\,\,\,\,\,\,\,\,\,P_{i+1}:=P_i^{\overline n_id_i}
+\sum_{t=0}^{d_i-1}\left((\sum_{s=1}^{\lambda_t}a_{s,t}P_0^{j_0(s,t)}P_1^{j_1(s,t)}\cdots P_{i-1}^{j_{i-1}(s,t)}\right)P_i^{t\overline n_i}
\end{equation}

$$
\begin{array}{ll}
&\mbox{ satisfies}\\ 
&\,\,\,\,\,\,\,\,\,\,b_{i,t}=\sum_{s=1}^{\lambda_t}\overline a_{s,t}\mbox{res}\left(\frac{P_0^{j_0(s,t)}P_1^{j_1(s,t)}\cdots P_{i-1}^{j_{i-1}(s,t)}}{U_i^{d_i-t}}\right)\\
&\mbox{ for $0\le t\le d_i-1$. } 
\end{array}
$$

\vskip .2truein
We inductively verify $\overline A(i)$ for $1\le i< \Lambda$ and the statements $B(i)$, $C(i)$ and $D(i)$ (with the residues
$[M]$ replaced with $\mbox{res}(M)$). We observe from $B(i)$ that the function res is determined by (\ref{eqL40}).
The inequality in 2) of the assumptions of the theorem is necessary to allow us to apply Lemma \ref{Lemma2}.

We now show that if $\Lambda=\infty$, then given $\sigma\in\ZZ_+$, there exists $\tau\in\ZZ_+$ such that
\begin{equation}\label{eqZ20}
\mbox{ord}(P_i)>\sigma\mbox{ if }i>\tau.
\end{equation}
We establish (\ref{eqZ20}) by induction on $\sigma$.  Suppose that $\mbox{ord}(P_i)>\sigma$ if $i>\tau$. There exists $\lambda$ such that $\beta_0<\beta_i$ if $i\ge\lambda$.
Let $\tau'=\max\{\sigma+\tau+1,\tau+1,\lambda\}$. We will show that $\mbox{ord}(P_i)>\sigma+1$ if $i>\tau'$. From (\ref{eqM1*}), we must show that if $i>\tau'$ and $(a_0,\ldots, a_{i-1})\in\NN^i$ are such that 
$$
a_0\beta_0+a_1\beta_1+\cdots+a_{i-1}\beta_{i-1}=n_{i-1}\beta_{i-1}
$$
then
\begin{equation}\label{eqZ21}
a_0\mbox{ord}(P_0)+a_1\mbox{ord}(P_1)+\cdots+a_{i-1}\mbox{ord}(P_{i-1})>\sigma+1.
\end{equation}
If $a_{\tau+1}+\cdots+a_{i-1}\ge 2$ then (\ref{eqZ21}) follows from induction. If $a_{\tau+1}+\cdots+a_{i-1}=1$
then some $a_j\ne 0$ with $0\le j\le \tau$ since $n_{i-1}>1$, so (\ref{eqZ21}) follows from induction. If $a_j=0$ for
$j\ge \tau+1$ then
$$
n_{i-1}\beta_{i-1}=a_0\beta_0+\cdots+a_{\tau}\beta_{\tau}<(a_0+\cdots+a_{\tau})\beta_{\tau}.
$$
Thus
$$
(a_0+\cdots+a_{\tau})>\frac{n_{i-1}\beta_{i-1}}{\beta_{\tau}}\ge 2^{i-\tau}>\sigma+1.
$$
Thus (\ref{eqZ21}) holds in this case.

We first suppose that for all $P_i$, there exists $m_i\in\ZZ_+$ such that $m_i\nu(P_i)>\min\{\beta_0,\beta_1\}$.

We  now establish the following:
\vskip .2truein
\noindent{\it Suppose that $f\in R$. Then there exists an expansion}
\begin{equation}\label{fund}
f=\sum_Ia_IP_0^{i_0}P_1^{i_1}\cdots P_r^{i_r}+\sum_J\phi_JP_0^{j_0}\cdots P_r^{j_r}
\end{equation}
{\it for some $r\in \NN$ where $\nu(P_0^{i_0}P_1^{i_1}\cdots P_r^{i_r})$ have a common value $\rho$ for all terms in the first sum,
 all $a_I\in CS$, $I,J\in \NN^{r+1}$ and some $a_I\ne 0$,  $0\le i_k< n_k$ for $1\le k\le r$
$\nu(P_0^{j_0}\cdots P_r^{j_r})>\rho$ for all terms in the second sum, and $\phi_J\in R$ for all terms in the second sum.
The first sum $\sum_Ia_IP_0^{i_0}P_1^{i_1}\cdots P_r^{i_r}$ is uniquely determined by these conditions.}
\vskip .2truein
The proof of (\ref{fund}) follows from the proofs of Lemma \ref{LemmaG10} and Theorem \ref{TheoremG2}, observing that all properties of
a valuation which $\nu$ is required to satisfy in these proofs hold for the function $\nu$ on Laurent monomials in the $P_i$ which we have defined above, and replacing $[M]$ in Lemma \ref{LemmaG10} with the function ${\rm res}(M)$ for Laurent monomials 
$M$ with $\nu(M)=0$.

The $n$ in the statement of Theorem \ref{TheoremG2} is chosen so that if $M$ is a monomial in the $P_i$ with
$\mbox{ord}(M)=\mbox{ord}(f)$, then $\nu(M)<n\min\{\beta_0,\beta_1\}$ (such an $n$ exists trivially if $\Lambda<\infty$ and by (\ref{eqZ20}) if $\Lambda=\infty$).

We can thus extend $\nu$ to $R$ by defining
$$
\nu(f)=\rho\mbox{ if $f$ has an expansion (\ref{fund})}.
$$
Now we will show that $\nu$ is  a valuation. 
Suppose that $f,g\in R$. We have expansions 
\begin{equation}\label{eqL20}
f=\sum_Ia_IP_0^{i_0}P_1^{i_1}\cdots P_r^{i_r}+\sum_J\phi_JP_0^{j_0}\cdots P_r^{j_r}
\end{equation}
and
\begin{equation}\label{eqL21}
g=\sum_Kb_KP_0^{k_0}P_1^{k_1}\cdots P_r^{k_r}+\sum_L\phi_LP_0^{l_0}\cdots P_r^{l_r}
\end{equation}
of the form (\ref{fund}). Let $\rho=\nu(f)$ and $\rho'=\nu(g)$.
The statement that $\nu(f+g)\ge \min\{\nu(f),\nu(g)\}$ follows from Remark \ref{RemarkG4} and the algorithm of Theorem \ref{TheoremG2}.

Let $V$ be a monomial in $P_0,\ldots, P_r$ such that $\nu(V)=\nu(P_0^{i_0}\cdots P_r^{i_r})$ for all $I$ in the first sum
of $f$ in (\ref{eqL20})
and let $W$ be a monomial in $P_0,\ldots, P_r$ such that $\nu(W)=\nu(P_0^{k_0}\cdots P_r^{k_r})$ for all $K$ in the first sum
of $g$ in (\ref{eqL21}).
We have that
$$
\sum \overline a_I\mbox{ res}\left(\frac{P_0^{i_0}\cdots P_r^{i_r}}{V}\right)\ne 0 \mbox{ in }L
$$
and
$$
\sum \overline b_K\mbox{ res}\left(\frac{P_0^{k_0}\cdots P_r^{k_r}}{W}\right)\ne 0 \mbox{ in }L
$$
by $D(r)$.

 We have (applying Remark \ref{RemarkG4}) an expansion
\begin{equation}\label{eqL22}
fg=\sum_Md_MP_0^{m_0}P_1^{m_1}\cdots P_r^{m_r}+\sum_Q\psi_QP_0^{q_0}\cdots P_r^{q_r}
\end{equation}
with $d_M\in S$ for all $M$, $\psi_Q\in R$ for all $Q$, $\nu(P_0^{m_0}P_1^{m_1}\cdots P_s^{m_r})=\rho+\rho'$ for all terms in the first sum, and some $d_M\ne 0$ and $\nu(P_0^{q_0}\cdots P_r^{q_r})>\rho+\rho'$ for all terms in the second sum, 
which satisfies all conditions of (\ref{fund}) except that we only have that $m_0,m_1,\ldots,m_r\in\NN$.
We have
$$
\sum_M\overline d_M\mbox{ res}\left(\frac{P_0^{m_0}\cdots P_r^{m_r}}{VW}\right)
=\left(\sum_I\overline a_I\mbox{ res}\left(\frac{P_0^{i_0}\cdots P_r^{i_r}}{V}\right)\right)
\left(\sum_K\overline b_K\mbox{ res}\left(\frac{P_0^{k_0}\cdots P_r^{k_r}}{W}\right)\right)\ne 0.
$$
By (\ref{eqW2}) of Lemma \ref{LemmaG10} (with $[M]$ replaced with $\mbox{res}(M)$ for a Laurent monomial $N$ with $\nu(M)=0$)
we see that the algorithm of Theorem \ref{TheoremG2} which puts the expansion (\ref{eqL22}) into the form (\ref{fund}) converges to an expression (\ref{fund}) where the terms in the first sum all have
$\nu(P_0^{i_0}\cdots P_r^{i_r})=\rho+\rho'$ with
$$
\sum_I\overline a_I\mbox{ res}\left(\frac{P_0^{i_0}\cdots P_r^{i_r}}{VW}\right)=\sum_M\overline d_M\mbox{ res}\left(\frac{P_0^{m_0}\cdots P_r^{m_r}}{VW}\right)\ne 0.
$$
Thus $\nu(fg)=\nu(f)+\nu(g)$. 
We have established that $\nu$ is a valuation.

By Theorem \ref{Corollary1*} or Case 1 of Theorem \ref{Corollary3*}, we have that $S=S^R(\nu)$ and $L=V_{\nu}/\mathfrak m_{\nu}$.
\vskip .2truein
Finally, we suppose that $\Lambda$ is finite and $\overline n_{\Lambda}=\infty$. Given $g\in R$, write
\begin{equation}\label{eqL31}
g=P_{\Lambda}^tf
\end{equation}
where $P_{\Lambda}\not\,\mid f$. Choose $n\in\ZZ_+$ so that if $M$ is a monomial in $P_0,\ldots, P_{\Lambda-1}$ with ${\rm ord}(M)={\rm ord}(f)$
then $\nu(M)<n\min\{\beta_0,\beta_1\}$.

The argument giving the expansion (\ref{fund}) now provides an expansion
\begin{equation}\label{eqL30}
f=\sum_Ia_IP_0^{i_0}\cdots P_{\Lambda}^{i_{\Lambda}}+\sum_J\phi_JP_0^{j_0}\cdots P_{\Lambda}^{j_{\Lambda}}+h_1
\end{equation}
where $\nu(P_0^{i_0}\cdots P_{\Lambda}^{i_{\Lambda}})$ has a common value $\rho$ for all monomials in the first sum, $a_I\in CS$ for all $I$, $\nu(P_0^{j_0}\cdots P_{\Lambda}^{j_{\Lambda}})>\rho$ for all monomials in the second sum, $\phi_J\in R$ for all $J$ and $h_1\in \mathfrak m_R^n$.

If $i_{\Lambda}=0$ for all monomials in the first sum, then we obtain an expansion of $f$ of the form (\ref{fund}). Suppose that  $i_{\Lambda}\ne 0$ for some monomial in the first sum. Then $i_{\Lambda}\ne 0$ for all terms in the first sum,   $j_{\Lambda}\ne 0$ for all terms in the second sum, and we
have an expression
$f=P_{\Lambda}t_1+h_1$  for some $t_1\in R$. Repeating this argument for increasingly large values of $n$, we either obtain an $n$ giving an expression (\ref{fund}) for $f$, or we obtain the statement that
$$
f\in \cap_{n=1}^{\infty}\left((P_{\Lambda})+\mathfrak m_R^n\right)=(P_{\Lambda}),
$$
which is impossible.
Thus we can extend $\nu$ to $R$ by defining $\nu(g)=t\beta_{\Lambda}+\rho$ if $g=P_{\Lambda}^tf$ where $P_{\Lambda}\not\,\mid f$ and
$f$ has an expansion (\ref{fund}).

It follows that $\nu$ is a valuation, by an extension of the proof of the previous case. By Case 2 of Theorem \ref{Corollary3*}, we have that
$S=S^R(\nu)$ and $L=V_{\nu}/\mathfrak m_{\nu}$.

\begin{Corollary}\label{CorollaryN1}
Suppose that $R$ is a regular local ring of dimension two and $\nu$ is a valuation dominating $R$. Then the semigroup $S^{R}(\nu)$ has a generating
set $\{\beta_i\}_{i\in I}$ and $V_{\nu}/\mathfrak m_{\nu}$ is generated over $\mathfrak k=R/\mathfrak m_R$ by a set $\{\alpha_i\}_{i\in I_+}$ such that 1) and 2) of Theorem \ref{Theorem3*} hold, but  the additional case that $\overline n_{\Lambda}<\infty$ and $d_{\Lambda}<\infty$ if $\Lambda<\infty$ may hold if $R$ is not complete.
\end{Corollary}

\begin{proof} The only case  we have not considered in Theorem \ref{Theorem3*} is the analysis in the case when 
$\Omega=\infty$, $n_i=1$ for $i\gg0$, 
$I_{\hat R}\ne 0$ and $I_{\hat R}\cap R= (0)$ (so that $R$ is not complete). In this case $\nu$ is discrete of rank 1, $\Lambda<\infty$, $\overline n_{\lambda}<\infty$ and $d_{\Lambda}<\infty$ by Remark \ref{Remark2}, giving the additional possibility stated
in the Corollary. 
\end{proof}

\section{Valuation Semigroups  on a regular local ring of dimension two}\label{Proof2}

In this section we prove Theorem \ref{Corollary4*} which is stated in the introduction. Theorem \ref{Corollary4*}
gives necessary and sufficent conditions for a semigroup  to be the valuation semigroup of a valuation dominating a regular local ring of dimension two.

\vskip .2truein
 If $S=S^{R}(\nu)$ for some valuation $\nu$ dominating $R$, then 1) and 2) of Theorem \ref{Corollary4*} hold by Corollary \ref{CorollaryN1}.
Observe that the construction in the proof of Theorem \ref{Theorem3*} of a valuation $\nu$  with a prescribed semigroup $S$ and residue field $L$ satisfying the conditions 1) and 2) of Theorem \ref{Theorem3*} is valid for any regular local ring $R$ of dimension 2 (with residue field $\mathfrak k$). 
Taking $L=\mathfrak k$ (or $L=\mathfrak k(t)$ where $t$ is an indeterminate), we  may thus construct a valuation $\nu$ dominating $R$ with semigroup $S^R(\nu)=S$ whenever $S$ satisfies the  conditions 1) and 2)  of Theorem \ref{Corollary4*}.

\begin{Definition}
Suppose that $S$ is a semigroup such that the group $G$ generated by $S$ is isomorphic to $\ZZ$. $S$ is  symmetric if there exists $m\in G$ such that $s\in S$ if and only if $m-s\not\in S$ for all $s\in G$.
\end{Definition}

We deduce from Theorem \ref{Corollary4*} a generalization of a result of Noh \cite{N}.

\begin{Corollary}\label{Symmetric} Suppose that $R$ is a regular local ring of dimension two and $\nu$ is a
 valuation dominating $R$ such that  $\nu$ is discrete of rank 1. Then $S^R(\nu)$ is symmetric.
 \end{Corollary}
 
 \begin{proof} By Theorem \ref{Corollary4*}, and since $\nu$ is discrete of rank 1, there exists a finite set
 $$
 \beta_0<\beta_1<\cdots<\beta_{\Lambda}
 $$
 such that $S^{\nu}(R)=S(\beta_0,\beta_1,\ldots,\beta_{\Lambda})$ and 
 $\beta_{i+1}>\overline n_i\beta_i$ for $1\le i<\Lambda$, where $\overline n_i=[G(\beta_0,\ldots,\beta_i):G(\beta_0,\ldots,\beta_{i-1})]$. We identify the value group $\Gamma_{\nu}$ with
 $\ZZ$. Then we calculate that
 $$
 {\rm lcm}\left({\rm gcd}(\beta_0,\ldots,\beta_{i-1}),\beta_i\right)=\overline n_i\beta_i
 $$
 for $1\le i\le \Lambda$. We have that $\overline n_i\beta_i\ge\beta_i>\overline n_{i-1}\beta_{i-1}$ for $2\le i\le \Lambda$. By Lemma \ref{Lemma2},
 we have that $\overline n_i\beta_i\in S(\beta_0,\ldots,\beta_{i-1})$ for $2\le i\le \Lambda$. Since $\beta_0$ and $\beta_1$ are both positive, we have that $\overline n_1\beta_1\in S(\beta_0)$. Thus
 the criteria of Proposition 2.1 \cite{H} is satisfied, so that $S^R(\nu)$ is symmetric.
 \end{proof}

\begin{Example} There exists  a semigroup $S$ which satisfies the sufficient conditions 1) and 2) of Theorem \ref{Corollary4*}, such that if $(R,\mathfrak m_R)$ is a 2-dimensional regular local ring dominated by a valuation $\nu$ such that $S^R(\nu)=S$, then $R/\mathfrak m_R=V_{\nu}/\mathfrak m_{\nu}$; that is, there can be no
residue field extension.
\end{Example}

\begin{proof} Define $\beta_i\in\QQ$ by
\begin{equation}\label{eqV1}
\beta_0=1,\,\beta_1=\frac{3}{2},\mbox{ and }\beta_i=2\beta_{i-1}+\frac{1}{2^i}\mbox{ for }i\ge 2.
\end{equation}

Let $S=S(\beta_0,\beta_1,\ldots)$ be the semigroup generated by $\beta_0,\beta_1,\ldots$. Observe that $\overline{n}_i = 2, \forall i \geq 1$, $\beta_0<\beta_1<\cdots$ is the minimal sequence of generators of $S$ and $S$ satisfies conditions 1) and 2) of Theorem \ref{Corollary4*}. The group $\Gamma=G(\beta_0,\beta_1,\ldots)$ generated by $S$ is $\Gamma=\frac{1}{2^{\infty}}\ZZ=\cup_{i=0}^{\infty}\frac{1}{2^i}\ZZ$.

Now suppose that $(R,\mathfrak m_R)$ is a regular local ring of dimension 2, with residue field $\mathfrak k$ and $\nu$ is a valuation of the quotient field of $R$ which dominates $R$ such that $S^R(\nu)=S$. Since $\Gamma_{\nu}=\frac{1}{2^{\infty}}$ is not discrete, 
we have by Proposition \ref{Prop17} that  $\nu$ extends uniquely to a valuation $\hat\nu$ of the quotient field of $\hat R$ which dominates $\hat R$ and $S^{\hat\nu}(\hat R)=S$.

We will now show that 
$V_{\nu}/\mathfrak m_{\nu}=V_{\hat\nu}/\mathfrak m_{\hat\nu}$. 
 Suppose that $f\in \hat R$. Since $\hat\nu$ has rank 1, there exists
a positive integer $n$ such that $\hat\nu(f)<n\nu(\mathfrak m)$. There exists $f'\in  R$ such that
$f''=f-f'\in \mathfrak m_R^n\hat R$. Thus $\nu(f)=\nu(f')$.
Suppose that $h\in  V_{\hat \nu}/\mathfrak m_{\hat \nu}$. Then $h=\left[\frac{f}{g}\right]$ where $f,g\in \hat R$ and
$\nu(f)=\nu(g)$. Write $f=f'+f''$ and $g=g'+g''$ where $f',g'\in R$ and $f'',g''\in \hat R$ satisfy $\nu(f'')>\nu(f)$
and $\nu(g'')>\nu(g)$. Then $\left[\frac{f}{g}\right]=\left[\frac{f'}{g'}\right]\in V_{\nu}/\mathfrak m_{\nu}$.

We also have $\mathfrak k = R / \mathfrak m_R = \hat R/ \mathfrak m_{\hat R}$. By Theorem \ref{Theorem3*}, there exists $\alpha_i \in V_{\hat v / M_{\hat v}}$ for $i \geq 1$ such that $V_{\hat v / M_{\hat v}} = \mathfrak k(\alpha_1, \alpha_2, ...)$ and if $d_i = [\mathfrak k(\alpha_1,..., \alpha_i) : \mathfrak k(\alpha_1,..., \alpha_{i-1})]$ then
\begin{equation}\label{eqV2}
\beta_{i+1} \geq \overline{n}_{i}d_i\beta_i, \forall i \geq 1,
\end{equation}
so that
\begin{equation}\label{eqV3}
[V_{\hat\nu}/\mathfrak m_{\hat\nu} : \mathfrak k] = \prod_{i=1}^{\infty}[\mathfrak k(\alpha_1,..., \alpha_i) : \mathfrak k(\alpha_1,..., \alpha_{i-1})] =\prod_{i=1}^{\infty}d_i.
\end{equation}
On the other hand, since $\beta_i \geq \beta_1 = \frac{3}{2}, \forall i \geq 1$, we have
\begin{equation}\label{eqV4}
\beta_{i+1} = 2\beta_i+\frac{1}{2^{i+1}} \leq 4\beta_i + \frac{1}{2^{i+1}} - 3 < 4\beta_i.
\end{equation}
 From \eqref{eqV2}, \eqref{eqV3} and  \eqref{eqV4} we have $d_i = 1,\forall i \geq 1$ so that $[V_{\hat\nu}/\mathfrak m_{\hat\nu} : \mathfrak k] = 1$.

\end{proof}

\section{Birational extensions}\label{RLR2}

Suppose that $R$ is a regular local ring of dimension two which is dominated by a valuation $\nu$. Let $\mathfrak k=R/\mathfrak m_R$. The quadratic transform $R_1$ of $R$ along $\nu$ is defined as follows. Let $u,v$ be a system of regular parameters in $R$,
where we may assume that $\nu(u)\le \nu(v)$. Then $R[\frac{v}{u}]\subset V_{\nu}$. Let 
$$
R_1=R\left[\frac{v}{u}\right]_{R[\frac{v}{u}]\cap \mathfrak m_{\nu}}.
$$
$R_1$ is a two dimensional regular local ring which is dominated by $\nu$. Let
\begin{equation}\label{eqX3}
R\rightarrow T_1\rightarrow T_2\cdots 
\end{equation}
be the sequence of quadratic transforms along $\nu$, so that $V_{\nu}=\cup T_i$ (\cite{Ab1}), and
$L=V_{\nu}/\mathfrak m_{\nu}=\cup T_i/\mathfrak m_{T_i}$.
Suppose that $x,y$ are regular parameters in $R$. 

\begin{Theorem}\label{birat}
Let $P_0=x$, $P_1=y$ and $\{P_i\}$ be the sequence of elements of $R$ constructed in Theorem \ref{Theorem1*}. Suppose that $\Omega\ge 2$. Then there exists some smallest value $i$ in the sequence (\ref{eqX3}) such that
the divisor of $xy$ in $\mbox{Spec}(T_i)$ has only one component. Let $R_1=T_i$.
Then $R_1/\mathfrak m_{R_1}\cong \mathfrak k(\alpha_1)$, and there exists $x_1\in R_1$ and $w\in\ZZ_+$ such that
$x_1=0$ is a local equation of the exceptional divisor of $\mbox{Spec}(R_1)\rightarrow \mbox{Spec}(R)$, and $Q_0=x_1$, $Q_1=\frac{P_2}{x_1^{wn_1}}$ are regular parameters in $R_1$. We have that
$$
Q_i=\frac{P_{i+1}}{Q_0^{w n_1\cdots n_i}}
$$
for $1\le i< \max\{\Omega,\infty\}$
satisfy the conclusions  of Theorem \ref{Theorem1*} (as interpreted by Remark \ref{RemarkH1}) for the ring $R_1$.
\end{Theorem}

\begin{proof} We use the notation of Theorem \ref{Theorem1*} and its proof for $R$ and the $\{P_i\}$. Recall that $U_1=U^{w_0(1)}$. Let $w=w_0(1)$. Since $\overline n_1$ and $w$ are relatively prime, there exist $a,b\in\NN$ such that 
$$
\epsilon:=\overline n_1b-wa=\pm 1.
$$
 Define
elements of the quotient field of $R$ by
$$
x_1=(x^by^{-a})^{\epsilon}, y_1=(x^{-w}y^{\overline n_1})^{\epsilon}.
$$
 We have that
\begin{equation}\label{eqZ1}
x=x_1^{\overline n_1}y_1^a, y=x_1^wy_1^b.
\end{equation}
Since $\overline n_1\nu(y)=w\nu(x)$, it follows that
$$
\overline n_1\nu(x_1)=\nu(x), \overline\nu(y_1)=0.
$$
We further have that
\begin{equation}\label{eqZ3}
\alpha_1=[y_1]^{\epsilon}\in L.
\end{equation}
Let $A=R[x_1,y_1]\subset V_{\nu}$ and $\mathfrak m_A=\mathfrak m_{\nu}\cap A$. $R\rightarrow A_{\mathfrak m_A}$ factors as a product
of quadratic transforms such that $xy$ has two distinct irreducible factors in all intermediate rings. Thus $A=R_1$. 
Recall that 
$$
f_1(u)=u^{d_1}+b_{1,d_1-1}u^{d-1-1}+\cdots+b_{1,0}
$$
is the minimal polynomial of $\alpha_1=\left[\frac{y^{\overline n_1}}{x^w}\right]$ over $\mathfrak k$, and from (\ref{eqM1}) of $A(1)$,
\begin{equation}\label{eqZ2}
P_2= y^{\overline n_1d_1}+a_{1,d_1-1}x^{w}y^{\overline n_1(d_1-1)}+\cdots+a_{1,0}x^{d_1w}.
\end{equation}
Substituting (\ref{eqZ1}) into (\ref{eqZ2}), we find that
$$
P_2=x_1^{wn_1}\left(y_1^{b\overline n_1 d_1}+a_{1,d_1-1}y_1^{aw+b\overline n_1 (d_1-1)}+\cdots +a_{1,0}y_1^{ad_1w}\right).
$$
Thus
$$
Q_1=\frac{P_2}{x_1^{wn_1}}\in R_1.
$$
We calculate
\begin{equation}\label{eqZ4}
\nu(Q_1)=\nu(P_2)-wn_1\nu(x_1)=\nu(P_2)-n_1\nu(P_1)>0
\end{equation}
 Thus $x_1,Q_1\in \mathfrak m_{R_1}$. 
 
 Suppose that $\epsilon=1$. Then since 
 $$
 Q_1=y_1^{awd_1}\left( y_1^{d_1}+a_{1,d_1-1}y_1^{d_1-1}+\cdots+a_{1,0}\right)
 $$
 and 
 $y_1$ is a unit in $R_1$, we have that
 $$
R_1/(x_1,Q_1)\cong  \mathfrak k[y_1]/(f(y_1))\cong\mathfrak k(\alpha_1).
$$

Suppose that $\epsilon=-1$. Let
$$
h(u)=y_1^{d_1}+\frac{b_{1,1}}{b_{1,0}}y_1^{d_1-1}+\cdots +\frac{1}{b_{1,0}},
$$
which is the minimal polynomial of $\alpha_1^{-1}$ over $\mathfrak k$.
Since 
 $$
 Q_1=y_1^{b\overline n_1d_1}\left( 1+a_{1,d_1-1}y_1+\cdots+a_{1,0}y_1^{d_1}\right)
 $$
 and 
 $y_1$ is a unit in $R_1$, we have that
 $$
R_1/(x_1,Q_1)\cong  \mathfrak k[y_1]/(h(y_1))\cong\mathfrak k(\alpha_1^{-1})=\mathfrak k(\alpha_1).
$$

Now define $\beta_i= \nu(P_i)$ and $\hat\beta_i=\nu(Q_i)$ for $i\ge 0$. 
We have 
$$
\hat\beta_i=\nu(P_{i+1})-wn_1\cdots n_i\hat\beta_0
$$
for $i\ge 1$.

Since $\mbox{gcd}(w,\overline n_1)=1$, we have that $G(\hat\beta_0)=G(\beta_0,\beta_1)$.
Thus 
$$
\overline n_{i+1}=[G(\hat\beta_0,\ldots, \hat \beta_i):G(\hat\beta_0,\ldots, \hat\beta_{i-1})]
$$
for $i\ge 1$.

 We will leave the proof that the analogue of $A(1)$ of Theorem \ref{Theorem1*} holds for  $Q_1$ in $R_1$ for the reader, as is an easier variation of the following
inductive statement, which we will prove.

Assume that $2\le i<\Omega-1$ and the analogue of $A(j)$ of Theorem \ref{Theorem1*} holds for $Q_j$ in $R_1$ for $j< i$. We will prove that the analogue of $A(i)$ of Theorem \ref{Theorem1*} holds for  $Q_{i}$ in $R_1$.

In particular, we assume that
\begin{equation}\label{eqJ1}
\hat\beta_{j+1}>n_{j+1}\hat\beta_j
\end{equation} 
for $1\le j\le i-1$.

Define
\begin{equation}\label{eqJ6}
\begin{array}{lll}
V_i&=&U_{i+1}Q_0^{-wn_1n_2\cdots n_i\overline n_{i+1}}y_1^{-(aw_{0}(i+1)+bw_{1}(i+1))}\\
&=& Q_0^{\hat w_{0}(i+1)}Q_1^{w_{2}(i+1)}\cdots Q_{i-1}^{w_{i}(i+1)}
\end{array}
\end{equation}
where
$$
\hat w_{0}(i+1)=\overline n_1w_{0}(i+1)+ww_{1}(i+1)+wn_1w_{2}(i+1)+\cdots +wn_1n_2\cdots n_{i-1}w_{i}(i+1)
-wn_1n_2\cdots n_i\overline n_{i+1}.
$$
We have that 
$$
\nu(Q_{i}^{\overline n_{i+1}})=\overline n_{i+1}\hat\beta_i = \overline n_{i+1}\nu(P_{i+1})-wn_1n_2\cdots n_i\overline n_{i+1}\hat \beta_0=\nu(V_i).
$$
Thus
$$
\overline n_{i+1}\hat \beta_i=\hat w_{0}(i+1)\hat \beta_0+\hat w_{2}(i+1)\hat \beta_1+\hat w_{3}(i+1)\hat\beta_2+\cdots
+w_{i}(i+1)\hat\beta_{i-1}.
$$
Recall that $0\le w_{j}(i+1)<n_j$ for $1\le j\le i$ and apply (\ref{eqJ1}) to obtain
\begin{equation}\label{eqJ2}
\begin{array}{lll}
\hat w_{0}(i+1)\hat\beta_0&=& \overline n_{i+1}\hat\beta_i-w_{i}(i+1)\hat\beta_{i-1}-\cdots
-w_{3}(i+1)\hat\beta_2-w_{2}(i+1)\hat\beta_1\\
&\ge&  \hat\beta_i-(n_i-1)\hat\beta_{i-1}-\cdots -(n_3-1)\hat\beta_2-(n_2-1)\hat\beta_1\\
&>& \hat\beta_i-(n_i-1)\hat\beta_{i-1}-\cdots -(n_4-1)\hat\beta_3-n_3\hat\beta_2\\
&&\vdots\\
&\ge & \hat\beta_i-n_i\hat\beta_{i-1}>0.
\end{array}
\end{equation} 
Thus $V_i\in R_1$.
We have
\begin{equation}\label{eqJ3}
\frac{Q_i^{\overline n_{i+1}}}{V_i}=\left(\frac{P_{i+1}^{\overline n_{i+1}}}{U_{i+1}}\right)y_1^{aw_{0}(i+1)+bw_{1}(i+1)}.
\end{equation}
Let
\begin{equation}\label{eqJ4}
\hat\alpha_i=\left[\frac{Q_i^{\overline n_{i+1}}}{V_i}\right]=\alpha_{i+1}\alpha_1^{\epsilon(aw_{0}(i+1)+bw_{1}(i+1))}\in L
\end{equation}
From the minimal  polynomial $f_{i+1}(u)$ of $\alpha_{i+1}$, we see that
$$
g_i(u)=u^{d_{i+1}}+b_{i+1,d_{i+1}-1}\alpha_1^{\epsilon(aw_{0}(i+1)+bw_{1}(i+1))d_{i+1}}u^{d_{i+1}-1}+\cdots +b_{i+1,0}\alpha_1^{\epsilon(aw_{0}(i+1)+bw_{1}(i+1))d_{i+1}}
$$
is the minimal polynomial of $\hat\alpha_i$ over $\mathfrak k(\alpha_1)(\hat\alpha_1,\ldots,\hat\alpha_{i-1})$.

Now from  equation (\ref{eqM1}) of $A(i+1)$ determining $P_{i+1}$, we obtain
\begin{equation}\label{eqJ5}
\begin{array}{lll}
Q_{i+1}&=& \frac{P_{i+2}}{Q_0^{wn_1n_2\cdots n_{i+1}}}\\
&=& Q_i^{\overline n_{i+1}d_{i+1}}+\sum_{t=0}^{d_{i+1}-1}\left(\sum_{s=1}^{\lambda_t}a_{s,t}y_1^{aj_0(s,t)+bj_1(s,t)}
Q_0^{\hat j_0(s,t)}Q_1^{j_2(s,t)}\cdots Q_{i-1}^{j_i(s,t)}\right)Q_i^{t\overline n_{i+1}}
\end{array}
\end{equation}
where
$$
\hat j_0(s,t)=\overline n_1j_0(s,t)+wj_1(s,t)+wn_1j_2(s,t)+\cdots+ wn_1n_2\cdots n_{i-1}j_{i}(s,t)-(d_{i+1}-t)wn_1n_2\cdots n_{i}\overline n_{i+1}.
$$
Recall that $0\le j_k(s,t)<n_k$ for $1\le k\le i$. We further have that
$$
\nu(Q_0^{\hat j_0(s,t)}Q_1^{j_2(s,t)}\cdots Q_{i-1}^{j_i(s,t)})=(d_{i+1}-t)\overline n_{i+1}\hat\beta_i\ge \hat \beta_i.
$$
By a similar argument to (\ref{eqJ2}), we obtain that $\hat j_0(s,t)>0$ for all $s,t$.

By the definition of $Q_{i+1}$, (\ref{eqJ6}) and (\ref{eqJ5}), we have
\begin{equation}\label{eqJ10}
\begin{array}{l}
y_1^{(aw_{0}(i+1)+bw_{1}(i+1))d_{i+1}}\frac{P_{i+2}}{U_{i+1}^{d_{i+1}}} = \frac{Q_{i+1}}{V_i^{d_{i+1}}}\\
\,\,\,\,\,= \left(\frac{Q_{i}^{\overline n_{i+1}}}{V_i}\right)^{d_{i+1}}
+\sum_{t=0}^{d_{i+1}-1}\left(\sum_{s=1}^{\lambda_t}y_1^{aj_0(s,t)+bj_1(s,t)}
\frac{Q_0^{\hat j_0(s,t)}Q_1^{j_2(s,t)}\cdots Q_{i-1}^{j_i(s,t)}}{V_i^{d_{i+1}-t}}\right)\left(\frac{Q_i^{\overline n_{i+1}}}{V_i}\right)^t
\end{array}
\end{equation}
We have
$$
\begin{array}{l}
\left[\sum_{s=1}^{\lambda_t}a_{s,t}y_1^{aj_0(s,t)+bj_1(s,t)}\frac{Q_0^{\hat j_0(s,t)}Q_1^{j_2(s,t)}\cdots Q_{i-1}^{j_i(s,t)}}{V_i^{d_{i+1}-t}}\right]\\
\,\,\,\,\,=\left[\sum_{s=1}^{\lambda_t}a_{s,t}y_1^{(aw_{0}(i+1)+bw_{1}(i+1))(d_{i+1}-t)}\frac{P_0^{j_0(s,t)}P_1^{j_1(s,t)}\cdots P_i^{j_i(s,t)}}
{U_{i+1}^{d_{i+1}-t}}\right]\\
\,\,\,\,\,=b_{i+1,t}\alpha_1^{\epsilon(aw_{0}(i+1)+bw_{1}(i+1))(d_{i+1}-t)}
\end{array}
$$
for $0\le t\le d_{i+1}-1$ and
$$
\left[\frac{Q_{i+1}}{V_i^{d_{i+1}}}\right]=g_i(\hat\alpha_i)=0.
$$
Thus 
$$
\begin{array}{lll}
\hat\beta_{i+1}&=&\nu(Q_{i+1})>d_{i+1}\nu(V_i)=d_{i+1}(\nu(U_{i+1})-wn_1n_2\cdots n_i\overline n_{i+1}\hat\beta_0)\\
&=&n_{i+1}(\nu(P_{i+1})-wn_1n_2\cdots n_i\hat\beta_0)=n_{i+1}\hat\beta_i.
\end{array}
$$
We have thus established that $A(i)$ holds for $Q_{i}$ in $R_1$. By induction on $i$, we have that $A(i)$ of Theorem
\ref{Theorem1*} holds for $Q_i$ in $R_1$ for $1\le i<\Omega-1$.
\vskip .2truein
We now will show that $D(r)$ of Theorem \ref{Theorem1*} holds for the $Q_i$ in $R_1$ for  all $r$. We begin by establishing the following statement:
\vskip .2truein
\noindent{\it Suppose that $\lambda\ge n_1w$ is as integer. Then there exist $\delta_0,\delta_1\in\NN$ with $0\le\delta_1<\overline n_1$ such that}
\begin{equation}\label{eqZ5}
x^{\delta_0+iw}y^{\delta_1+(d_1-1-i)\overline n_1}=x_1^{\lambda}y_1^{z-i\epsilon}
\end{equation}
\noindent{\it for $0\le i\le d_1-1$ where $z=a\delta_0+b(\delta_1+(d_1-1)\overline n_1$.}
\vskip .2truein
We first  prove (\ref{eqZ5}). We have that 
$$
(\lambda\epsilon b -rw)\overline n_1+(r\overline n_1-\lambda\epsilon a)w=\lambda
$$
for all $r\in\ZZ$. Choose $r$ so that $\delta_1=r\overline n_1-\lambda\epsilon  a$ satisfies $0\le \delta_1<\overline n_1$.
Set 
$$
\delta_0=(\lambda\epsilon b-rw)-(d_1-1)w.
$$
 Then
$$
(\lambda \epsilon b-rw)\overline n_1=\lambda-\delta_1w\ge n_1w-(\overline n_1-1)w=(n_1-\overline n_1+1)w
$$
so
$$
\delta_0\ge (n_1-\overline n_1-d_1+2)w=\left((\overline n_1-1)(d_1-1)+1\right)w\ge w.
$$
Substituting (\ref{eqZ1}) in $x^{\delta_0+iw}y^{\delta_1+(d_1-1-i)\overline n_1}$, we obtain the formula (\ref{eqZ5}).
\vskip .2truein
We now will prove that statement $D(r)$  of Theorem \ref{Theorem1*} holds  for the $Q_i$ in $R_1$  for all $r$.

Suppose that we have monomials $Q_0^{j_0(l)}Q_1^{j_1(l)}\cdots Q_r^{j_r(l)}$ for $1\le l \le m$ such that
$$
\nu(Q_0^{j_0(l)}Q_1^{j_1(l)}\cdots Q_r^{j_r(l)})=\nu(Q_0^{j_0(1)}Q_1^{j_1(1)}\cdots Q_r^{j_r(1)})
$$
 for $1\le l\le m$, and that we
have a dependence relation in $L=V_{\nu}/\mathfrak m_{\nu}$.
$$
0=e_1+e_2\left[\frac{Q_0^{j_0(2)}Q_1^{j_1(2)}\cdots Q_r^{j_r(2)}}{Q_0^{j_0(1)}Q_1^{j_1(1)}\cdots Q_r^{j_r(1)}}\right]
+\cdots+e_m\left[\frac{Q_0^{j_0(m)}Q_1^{j_1(m)}\cdots Q_r^{j_r(m)}}{Q_0^{j_0(1)}Q_1^{j_1(1)}\cdots Q_r^{j_r(1)}}\right]
$$
with $e_i\in\mathfrak k(\alpha_1)$ (and some $e_i\ne 0$). Multiplying the $Q_0^{j_0(l)}Q_1^{j_1(l)}\cdots Q_r^{j_r(l)}$ for $1\le l\le m$ by a common term $Q_0^t$ with $t$
a sufficiently large positive integer, we may assume that 
$$
\hat j_0(l)=j_0(l)-j_1(l)wn_1-j_2(l)wn_1n_2-\cdots -j_r(l)wn_1n_2\cdots n_r\ge n_1w
$$
for $1\le l\le m$. We have that
$$
Q_0^{j_0(l)}Q_1^{j_1(l)}\cdots Q_r^{j_r(l)}=Q_0^{\hat j_0(l)}P_2^{j_1(l)}\cdots P_{r+1}^{j_r(l)}.
$$
Since $\hat j_0(l)\ge wn_1$, (\ref{eqZ5}) implies that for each $l$ with $1\le l\le w$, there exist $\delta_0(l), \delta_1(l)$ with
$\delta_0(l),\delta_1(l)\in\NN$ and $0\le\delta_1(l)<\overline n_1$ such that
$$
P_0^{\delta_0(l)+iw}P_1^{\delta_1(l)+(d_1-1-i)\overline n_1}=y_1^{z(l)-i\epsilon}Q_0^{\hat j_0(l)}
$$
for $0\le i\le d_1-1$. The ordered set
$$
\{ \alpha_1^{\epsilon(z(l)-z(1))}, \alpha_1^{\epsilon(z(l)-z(1))-1},\cdots, \alpha_1^{\epsilon(z(l)-z(1))-(d_1-1)}\}
$$
is a $\mathfrak k$-basis of $\mathfrak k(\alpha_1)$ for all $l$ (since multiplication by $\alpha_1^{\epsilon(z(l)-z(1))+(d_1-1)}$ is a
$\mathfrak k$-vector space isomorphism of $\mathfrak k(\alpha_1)$, and thus takes a basis to a basis). Thus there exists
$e_{l,i}\in \mathfrak k$ such that
$$
e_l=\sum_{i=0}^{d_1-1}e_{l,i}\alpha_1^{\epsilon(z(l)-z(1))-i}. 
$$
Since some $e_{l,i}\ne 0$, we have a dependence relation 
$$
0=\sum_{l=1}^m\sum_{i=0}^{d_1-1}e_{l,i}\left[\frac{P_0^{\delta_0(l)+iw}P_1^{\delta_1(l)+(d_1-1-i)\overline n_1}P_2^{j_1(l)}\cdots
P_{r+1}^{j_r(l)}}
{P_0^{\delta_0(1)}P_1^{\delta_1(1)+(d_1-1)\overline n_1}P_2^{j_1(1)}\cdots
P_{r+1}^{j_r(1)}}\right],
$$
a contradiction to $D(r+1)$ of Theorem \ref{Theorem1*} for the $P_i$ in $R$. Thus we have established $D(r)$ of Theorem \ref{Theorem1*} for the $Q_i$ in $R_1$.
\end{proof}

\section{Polynomial rings in two variables}\label{Poly}

The algorithm of Theorem \ref{Theorem1*} is applicable when $R=\mathfrak k[x,y]$ is a polynomial ring over a field
and $\nu$ is a valuation which dominates the maximal ideal $(x,y)$ of $R$. In this case many of the calculations in this paper
become much simpler, as we now indicate (of course we take the coefficient set $CF$ to be the field $\mathfrak k$). In the case when $R$ is equicharacteristic, we can establish from the polynomial case
 the results of this paper using Cohen's structure theorem and Proposition \ref{Prop17} to reduce to the case of a polynomial ring
 in two variables.

 If $f\in R=\mathfrak k[x,y]$ is a nonzero polynomial, then we have an expansion
 $f=a_0(x)+a_1(x)y+\cdots +a_r(x)y^r$ where $a_i(x)\in\mathfrak k[x]$ for all $i$ and $a_r(x)\ne 0$. We define $\mbox{ord}_y(f)=r$,
 and say that $f$ is monic in $y$ if $a_r(x)\in\mathfrak k$. 
 We first establish the following formula. 
 \begin{equation}\label{eqZ60}
 \mbox{$P_i$ is monic in $y$ with }
 \mbox{deg}_yP_i=n_1n_2\cdots n_{i-1}\mbox{ for }i\ge 2.
 \end{equation}
 
 We establish (\ref{eqZ60}) by induction. In the expansion (\ref{eqM1}) of $P_{i+1}$, we have for $0\le t\le d_i-1$ and whenever $a_{s,t}\ne 0$, that $0\le j_k(s,t)<n_k$ for $1\le k\le i-1$. Thus
 $$
 \begin{array}{l}
 \mbox{deg}_y(P_0^{j_0(s,t)}P_1^{j_1(s,t)}\cdots P_{i-1}^{j_{i-1}(s,t)}P_i^{t\overline n_i})\\
 \,\,\,\,\,=j_1(s,t)+j_2(s,t)n_1+j_3(s,t)n_1n_2+\cdots +j_{i-1}n_1n_2\cdots n_{i-2}+t\overline n_in_1n_2\cdots n_{i-1}\\
 \,\,\,\,\,<n_1n_2\cdots n_i.
 \end{array}
 $$
 Thus $\mbox{deg}_yP_{i+1}=\mbox{deg}_yP_i^{n_i}=n_1n_2\cdots n_i$. We further see that $P_{i+1}$ is monic in $y$.
 
 Set $\sigma(0)=0$ and for $i\ge 1$ let 
 $$
 \sigma(i)=\min\{j\mid j>\sigma(i-1)\mbox{ and }n_j>1\}.
 $$
 Let $Q_i=P_{\sigma(i)}$. 
 We calculate (as long as we are not in the case $\Omega=\infty$ and $n_i=1$ for $i\gg 0$) that for $d\in\ZZ_+$, there exists
 a unique $r\in\ZZ_+$ and $j_1,\ldots, j_r\in\ZZ_+$ such that $0\le j_k<n_k$ for $1\le k\le r$ and
 $\mbox{deg}_y Q_1^{j_1}\cdots Q_r^{j_r}=d$. Let $M_d$ be this monomial. Since the monomials $M_d$ are monic in $y$,
 we see (continuing to assume that we are not in the case $\Omega=\infty$ and $n_i=1$ for $i\gg 0$) that if $f\in R=\mathfrak k[x,y]$ is nonzero with $\mbox{deg}_y(f)=d$, then there is a unique expression
 \begin{equation}\label{eqZ61}
 f=\sum_{i=0}^dA_i(x)M_i
 \end{equation}
 where $A_i(x)\in\mathfrak k[x]$, and
 \begin{equation}\label{eqZ62}
 \nu(f)=\min_i\{\mbox{ord}(A_i)\nu(Q_0)+\nu(M_i)\}.
 \end{equation}
 
 In the case when  $\Omega=\infty$ and $n_i=1$ for $i\gg 0$ we have a similar statement, but we may need to introduce a new polynomial $g$ of ``infinite value'' as in  Case 3 of Theorem \ref{Corollary3*}.

\section{The $A_2$ singularity}\label{norm}

\begin{Lemma}\label{Lemma3}
Let $\mathfrak k$ be an algebraically closed field, and let $A=\mathfrak k[x^2,xy,y^2]$, a subring  of the polynomial ring $B=\mathfrak k[x,y]$. Let $\mathfrak m=(x^2,xy,y^2)A$  and  $\mathfrak n=(x,y)B$. Suppose that $\nu$ is a rational nondiscrete
valuation dominating $B_{\mathfrak n}$, such that $\nu$ has a  generating sequence 
$$
P_0=x,P_1=y,P_2,\ldots
$$
 in $\mathfrak k[x,y]$ of the form of the conclusions of Theorem \ref{Theorem1*},
such that each $P_i$ is a $\mathfrak k$-linear combinations of monomials in $x$ and $y$ of odd degree,
and
$$
\beta_0=\nu(x), \beta_1=\nu(y),\beta_2=\nu(P_2),\ldots
$$
is the increasing sequence of minimal generators of $S^{B_{\mathfrak n}}(\nu)$,
with $\beta_{i+1}>\overline n_i\beta_i$ for $i\ge 1$, where 
$$
\overline n_i=[G(\beta_0,\ldots,\beta_i):G(\beta_0,\ldots,\beta_{i-1})].
$$
Then
$$
S^{A_{\mathfrak m}}(\nu)=
\left\{
\begin{array}{l}
a_0\beta_0+a_1\beta_1+\cdots+a_i\beta_i\mid i\in \NN, a_0,\ldots,a_i\in\NN\\
\mbox{ and }a_0+a_1\cdots+a_i\equiv 0\mbox{ mod } 2
\end{array}\right\}.
$$
\end{Lemma}

\begin{proof} For $f\in \mathfrak k[x,y]$, let $t=\mbox{deg}_y(f)$. By (\ref{eqZ61}), $f$ has a unique expansion 
$$
f=\sum_{i=0}^t(\sum_k b_{k,i}x^k)P_1^{j_1(i)}\cdots P_r^{j_r(i)}
$$
where $b_{k,i}\in \mathfrak k$, $0\le j_k(i)<\overline n_k$ for $1\le k$ and 
$$
\mbox{deg}_yP_1^{j_1(i)}\cdots P_r^{j_r(i)}=i
$$
for all $i$. Looking first at the $t=\mbox{deg}_y(f)$ term, and then at lower order terms, we see that
$f\in \mathfrak k[x^2,xy,y^2]$ if and only if $k+j_1(i)+\cdots +j_r(i)\equiv 0 \mbox{ mod }2$ whenever $b_{k,i}\ne 0$.
\end{proof}

\begin{Example}\label{Example1} Suppose that $\mathfrak k$ is a field and $R$ is the localization of  $\mathfrak k[u,v,w]/uv-w^2$ at the maximal ideal $(u,v,w)$.  Then there exists a rational nondiscrete valuation $\nu$ dominating $R$ such that if
$$
\gamma_0<\gamma_1<\cdots
$$
is the increasing sequence of minimal generators of the semigroup $S^R(\nu)$,
then given $n\in\ZZ_+$, there exists $i>n$ such that $\gamma_{i+1}=\gamma_i+\frac{\gamma_0}{3}$ and
$\gamma_{i+1}$ is in the group generated by $\gamma_0,\ldots,\gamma_i$.
\end{Example}

\begin{proof} Let $A=\mathfrak k[x,y]$ be a polynomial ring with maximal ideal
$\mathfrak n=(x,y)\mathfrak k[x,y]$. We will use the criterion of Theorem \ref{Corollary4*} to construct a  rational nondiscrete valuation $\nu$ dominating $T=A_{\mathfrak n}$, with a generating sequence 
$$
P_0=x, P_0=y, P_2,\ldots
$$
such that 
$$
\beta_0=\nu(x), \beta_1=\nu(y), \beta_2=\nu(P_2),\ldots
$$
is the increasing set of minimal generators of the semigroup $S^{T}(\nu)$.
We will construct the $P_i$ so that each $P_i$ is a $\mathfrak k$-linear combination of
monomials in $x$ and $y$ of odd degree.

We define the first part of a generating sequence by setting
$$
P_0=x, P_1=y, P_2=y^3-x^5,
$$
with  $\beta_0=\nu(x)=1$, $\beta_1=\nu(y)=\frac{5}{3}$. Set $b_1=0$.

We now inductively define 
$$
P_{i+1}=P_i^3-x^{a_i}P_{i-1}
$$
with $a_i$ an even positive integer, and $\beta_i=\nu(P_i)=b_i+\frac{5}{3^i}$ with $b_i\in \ZZ_+$, for $i\ge 2$, by requiring that 3 divides $a_i+b_{i-1}$ and
$$
b_i=\frac{a_i+b_{i-1}}{3}>3b_{i-1}+5
$$
for $i\ge 2$. $a_i,b_i$ satisfying these relations can be constructed inductively from $b_{i-1}$.

Now let $B=\mathfrak k[x^2,xy,y^2]$, $\mathfrak m=(x^2,xy,y^2)B$, so that
$R\cong B_{\mathfrak m}$. With this identification, the semigroup 
$S^{R}(\nu)$ is generated by $\{\beta_i+\beta_j\mid i,j\in\NN\}$. From $3\beta_i<\beta_{i+1}$ for $i\ge 1$ and
$\beta_i<\beta_j$ if $j>i$, we conclude that if $i\le j$, $k\le l$ and $j< l$, then
\begin{equation}\label{eq10}
\beta_i+\beta_j< \beta_k+\beta_l.
\end{equation}

Let 
$$
\gamma_0=2<\gamma_1<\cdots
$$
be the sequence of minimal generators of the semigroup $S^{R}(\nu)$. By (\ref{eq10}), for $n\in\ZZ_+$, there
exists an index $l$ such that $\gamma_l=\beta_0+\beta_n$. We have $l\ge n$.
The semigroup $S(\gamma_0,\gamma_1,\ldots,\gamma_l)$  is generated by 
$$
\{\beta_i+\beta_j\mid i\le j\mbox{ and }j\le n-1\}
$$
and $\beta_0+\beta_n$.  

Suppose $\beta_1+\beta_n\in S(\gamma_0,\gamma_1,\ldots,\gamma_l)$.
Since $S(\gamma_0,\ldots, \gamma_{l-1})\subset \frac{1}{3^{n-1}}\ZZ$,
we have an expression
$\beta_1+\beta_n=r\gamma_l+\tau$ with $r$ a positive integer, and $\tau\in S(\gamma_0,\ldots,\gamma_{l-1})$. Now
$$
\gamma_l=\beta_0+\beta_n=1+b_n+\frac{5}{3^n}
$$
and
$$
\beta_1+\beta_n=\frac{5}{3}+b_n+\frac{5}{3^n}
$$
implies $\tau\le \frac{5}{3}-1=\frac{2}{3}$,
which is impossible, since $\gamma_0=\beta_0+\beta_0=2$. Thus $\beta_1+\beta_n\not\in S(\gamma_0,\gamma_1,\ldots,\gamma_l)$ and $\beta_1+\beta_n=\gamma_{l+1}$ is the next largest minimal generator of $S^{R}(\nu)$.

We have that $\gamma_{l+1}=\beta_1+\beta_n=(\beta_0+\beta_1)+(\beta_0+\beta_n)-2\beta_0$ is in the group generated by
$\gamma_0,\ldots,\gamma_l$.
\end{proof}

\begin{Example}\label{Example2} Let notation be as in Example \ref{Example1} and its proof. Then $R\rightarrow T$ is finite, but $S^{T}(\nu)$ is not a finitely generated $S^{R}(\nu)$ module.
\end{Example}

\begin{proof} Suppose $S^{T}(\nu)$ is a finitely generated $S^{R}(\nu)$ module. 
Then there exists $n>0$ such that $S^{T}(\nu)$ is generated by $\beta_0,\ldots,\beta_n$ and
$\{\beta_i+\beta_j\}\mid i,j\in\NN\}$. For $l>n$, $\beta_l$ cannot be in this semigroup.
\end{proof}

\begin{Example}\label{Example3} Let $A=\mathfrak k[u,v]_{(u,v)}$. Then $A\rightarrow T$ is a finite extension of regular local rings, but $S^{T}(\nu)$ is not
a finitely generated $S^{A}(\nu)$ module.
\end{Example}

\begin{proof} Since $A$ is a subring of $R$, $S^{A}(\nu)$ is a subsemigroup of $S^{R}(\nu)$. Since $S^{T}(\nu)$ is not a finitely generated $S^{R}(\nu)$-module, by Example \ref{Example2}, $S^{T}(\nu)$ cannot be a finitely generated $S^{A}(\nu)$-module.
\end{proof}


\begin{thebibliography}{1000000000}
\bibitem{Ab1} S. Abhyankar, On the valuations centered in a local domain, Amer. J. Math. 78 (1956), 321 - 348.
\bibitem{Ab2} S. Abhyankar, Ramification theoretic methods in algebraic geometry, Princeton Univ Press, 1959.
\bibitem{Ab3} S. Abhyankar, Newton-Puiseux expansion and generalized Tschirnhausen transformation I,
J. Reine Angew. Math. 260 (1973), 47 -83.
\bibitem{Ab4} S. Abhyankar, Newton-Puiseux expansion and generalized Tschirnhausen transformation II,
J. Reine Angew. Math. 261 (1973), 29-54.
\bibitem{B}K. Brauner, Klassification der singularit\"aten algebroider Kurven, 
Abh. math. semin. Hamburg. Univ 6 (1928).
\bibitem{BK} E. Brieskorn and H. Kn\"orrer, Plane algebraic curves,  Birkhau\"auser, (1986). 
\bibitem{Ca} A. Campillo, Algebroid curves in positive characteristic, Springer-Verlag, Berlin, Heidelberg, New York, 1980.
\bibitem{CoG} V. Cossart and G. Moreno-Soc\'ias, Racines approch\'ees, suites g\'en\'eratrices, sufficance des jets,
Ann. Fac. Sci. Toulouse math. (6) 14 (2005), 353-394.
\bibitem{CGP} V. Cossart, C. Galindo, O. Piltant, Un exemple effectif de gradu\'e non noeth\'erien associ\'e \`a une valuation divisorielle,
Ann. Inst. Fourier (Grenoble) 50 (2000), 105-112.
\bibitem{C} S.D. Cutkosky, Local factorization and monomialization of morphisms, Ast\'erisque
260, 1999.
\bibitem{CDK} S.D. Cutkosky, Kia Dalili and Olga Kashcheyeva,
Growth of rank 1 valuation semigroups,  Communications in Algebra 38 (2010), 2768 -- 2789.
\bibitem{CE} S.D. Cutkosky and S. El Hitti, Formal prime ideals of infinite value and their algebraic resolution,
to appear in Annales de la Facult\'e des Sciences de Toulouse, Mathem\'ematiques.
\bibitem{CG} S.D. Cutkosky and L. Ghezzi,  Completions of valuation rings,
Contemp. math. 386 (2005), 13 - 34.
\bibitem{CK} S.D. Cutkosky and O. Kashcheyeva, Algebraic series and valuation rings over nonclosed fields,
J. Pure. Appl. Alg. 212 (2008), 1996 - 2010.
\bibitem{CP} S.D. Cutkosky and O. Piltant, Ramification of Valuations, Advances in Math. 183 (2004), 1-79.
\bibitem{CS} S.D. Cutkosky and H. Srinivasan, The algebraic fundamental group of the complement of a curve singularity,
J. Algebra 230 (2000), 101 - 126.
\bibitem{CT1} S.D. Cutkosky and B. Teissier, Semigroups of valuations on local rings,
Mich. Math. J. 57 (2008), 173 - 193.
\bibitem{CT2} S.D. Cutkosky and B. Teissier, Semigroups of valuations on local rings II,
to appear in Amer. J. Math.
\bibitem{EN} D. Eisenbud and W. Neumann, Three dimensional link theory and invariants of plane curve singularities,
Ann. Math. Studies 110,
\bibitem{EH} S. El Hitti, A geometric construction of minimal generating sequences,  Master's Thesis, University of Missouri,
2006.
Princeton Univ. Press, Princeton, N.J. (1985)
\bibitem{FJ} C. Favre and M. Jonsson, The valuative tree, Lecture Notes in Math 1853, Springer
Verlag, Berlin, Heidelberg, New York, 2004.
\bibitem{GHK} L. Ghezzi, Huy T\`ai H\`a and O. Kashcheyeva, Toroidalization of generating sequences in dimension two
function fields, J. Algebra 301 (2006) 838-866.
\bibitem{GK} L. Ghezzi and O. Kashcheyeva, Toroidalization of generating sequences in dimension two
function fields of positive characteristic, J. Pure Appl. Algebra 209 (2007), 631-649.
\bibitem{GT} R. Goldin and B. Teissier, Resolving singularities of plane analytic branches with one toric morphism,
Valuation theory and its applications II,
F.V. Kuhlmann, S. Kuhlmann and M. Marshall, editors, Fields
Institute Communications 33, Amer. Math. Soc., Providence, RI, 315 -- 340.
\bibitem{GAS} F.J. Herrera Govantes, M.A. Olalla Acosta, M. Spivakovsky, Valuations in algebraic field extensions,
J. Algebra 312 (2007), 1033 - 1074.

\bibitem{G} A. Grothendieck, and A. Dieudonn\'e, El\'ements de g\'eom\'etrie alg\'ebrique IV, vol. 2, Publ. Math. IHES 24 (1965).
\bibitem{G-P} E.R. Garcia Barroso and P.D. Gonz\'alez-P\'erez, Decomposition in bunches of the critical locus of a quasi-ordinary map, Compos. math. 141 (2005) 461 -486.
\bibitem{HS} W. Heinzer, W. and J. Sally,  Extensions of valuations to
the completion of a local domain, Journal of Pure and Applied
Algebra 71 (1991), 175 - 185.
\bibitem{H} J. Herzog, Generators and relations of abelian semigroups and semigroup rings,
Manuscript math. 3 (1970), 175 - 193.
\bibitem{HOST} F.J. Herrera Govantes, M.S. Olalla Acosta, M. Spivakovsky, B. Teissier, 
Extending a valuation centered in a local domain to the formal completion,
arXiv:1007.4656.
\bibitem{Ka} E. K\"ahler, \"Uber die Verzweigung einer algebraischen Funktion zweier 
Ver\"anderlichen in der Umgebung einer singul\"aren Stelle. Math. Z. 30 (1929).
\bibitem{K} F.-V. Kuhlmann, Value groups, residue fields, and bad places of algebraic function fields, Trans. Amer. Math. Soc. 40 (1936), 363 - 395.
\bibitem{Li} J. Lipman, Proximity inequalities for complete ideals in two-dimensional regular local rings,
In: Commutative Algebra, Syzygies, Multiplicities and Birational Algebra (South Hadley M.A. 1992) Contemp. Math. 159 (1992), 293 - 306.
\bibitem{LMSS} F. Lucas, J. Madden, D. Schaub and M. Spivakovsky, Approximate roots of a valuation and the Pierce-Birkhoff conjecture, arXiv:1003.1180.
\bibitem{M} S. MacLane,  A construction for absolute values in polynomial rings, Trans. Amer.
Math. Soc. 40 (1936), 363 - 395.
\bibitem{MS} S. MacLane and O. Schilling,  Zero-dimensional branches of rank 1 on algebraic
varieties, Annals of Math. 40 (1939), 507 - 520.
\bibitem{Mi} J. Milnor, Singular points of complex hypersurfaces, Annals of Math. Studies 61
Princeton (1968). 
\bibitem{Mo} M. Moghaddam, A construction for a class of valuations of the field $K(X_1,\ldots, X_d,Y)$ with large value group, Journal of Algebra, 319, 7 (2008), 2803-2829.
\bibitem{Na} M. Nagata, Local Rings, John Wiley and Sons, New York (1962).
\bibitem{N} S. Noh, The value semigroup of prime divisors of the second kind in 2-dim regular local rings,
Tran. Amer. Math. Soc 336 (1993), 607 - 619.
\bibitem{S} M. Spivakovsky, Valuations in function fields of surfaces, Amer. J. Math. 112 (1990), 107 - 156.
 \bibitem{T} B. Teissier,   Valuations, deformations and toric geometry, Valuation theory and its applications II,
F.V. Kuhlmann, S. Kuhlmann and M. Marshall, editors, Fields
Institute Communications 33, Amer. Math. Soc., Providence, RI, 361
-- 459.
\bibitem{V} M. Vaqui\'e, Extension d'une valuation, Trans. Amer. Math. Soc. 359 (2007), 3439-3481.
\bibitem{Z3} O. Zariski, On the topology of algebroid singularities, Amer. J. Math., 54 (1932).
\bibitem{Z4}  O. Zariski, Algebraic Surfaces, (1935). Second supplemented edition, Ergebnisse der Math. 61,
 Springer Verlag, (1971). 
\bibitem{Z1} O. Zariski, Polynomial ideals defined by infinitely near base points, Amer. J. Math 60 (1938), 151 -204.
\bibitem{Z2} O. Zariski, The reduction of the singularities of an algebraic surface, Ann. Math. 40 (1939), 639 - 689.
\bibitem{ZS} O. Zariski and P. Samuel, Commutative Algebra Volume II, Van Nostrand, 1960.
\end{thebibliography}
\end{document}